\documentclass[a4paper,11pt]{article}

\usepackage{
amssymb,
amsmath,
amsthm,
eucal,
}

\makeatletter
 
  \@addtoreset{equation}{section}
 \makeatother
\theoremstyle{plain}

\newtheorem{theorem}{Theorem}[section]
\newtheorem{lemma}[theorem]{Lemma}
\newtheorem{proposition}[theorem]{Proposition}
\newtheorem{corollary}[theorem]{Corollary}
\theoremstyle{definition}

\newtheorem{remark}[theorem]{Remark}

\newtheorem{assumption}{Assumption}

\newcommand{\norm}[1]{{||#1||}}

\newcommand{\dderiv}[4]{{\partial_{#1}^{#3}\partial_{#2}^{#4}}}

\newcommand{\wtilde}[1]{{\widetilde{#1}}}
\def\supp{\mathop{\mathrm{supp}}\nolimits}

\def\Id{\mathop{\mathrm{Id}}\nolimits}

\def\Im{\mathop{\mathrm{Im}}\nolimits}
\def\Sym{\mathop{\mathrm{Sym}}\nolimits}
\def\R{{\mathbb{R}}}
\def\Z{{\mathbb{Z}}}
\def\N{{\mathbb{N}}}
\def\C{{\mathbb{C}}}

\def\S{{\mathcal{S}}}

\def\B{{\mathcal{B}}}

\def\X{{\mathcal{X}}}
\def\Y{{\mathcal{Y}}}

\def\<{{\langle}}
\def\>{{\rangle}}
\def\ep{{\varepsilon}}

\title
{Strichartz estimates for Schr\"odinger equations with variable coefficients and unbounded potentials II. Superquadratic potentials}
\author
{Haruya Mizutani${}^*$}
\date{\empty}
\begin{document}
\maketitle

\begin{abstract}
In this paper we prove local-in-time Strichartz estimates with loss of derivatives for Schr\"odinger equations with variable coefficients and potentials, under the conditions that the geodesic flow is nontrapping and potentials grow polynomially at infinity. This is a generalization to the case with variable coefficients and improvement of the result by Yajima-Zhang \cite{Yajima_Zhang_2}. The proof is based on microlocal techniques including the semiclassical parametrix for a time scale depending on a spatial localization and the Littlewood-Paley type decomposition with respect to both of space and frequency. 
\end{abstract}
\footnotetext{${}^*$Department of Mathematics, Graduate School of Science, Osaka University, Toyonaka, Osaka 560-0043, Japan. E-mail: \texttt{haruya@math.sci.osaka-u.ac.jp}. 
}


\section{Introduction}
Let $\wtilde P$ be a Schr\"odinger operator on $\R^d$ with variable coefficients $g^{jk}(x)$ and electromagnetic potentials $V(x)$ and $A(x)=(A_1(x),...,A_d(x))$ of the form:
$$
\wtilde P=\frac12(D_j-A_j(x))g^{jk}(x)(D_k-A_k(x))+V(x),\ D_j:=-i\partial/\partial x_j,\ x\in \R^d.
$$
Here and in the sequel, we use the standard summation convention. Then we impose the following conditions:

\begin{assumption}
\label{assumption_A}
$g^{jk},A_j,V\in C^\infty(\R^d;\R)$, and $(g^{jk}(x))_{j,k}$ is symmetric and uniformly elliptic:
$$
g^{jk}(x)\xi_j\xi_k \ge c|\xi|^2
$$ 
on $\R^{2d}$ with some positive constant $c>0$. Moreover, there exists $m\ge2$ such that, for any $\alpha\in \Z_+^d:=\N^d\cup\{0\}$,
\begin{align}
\label{assumption_A_1}
|\partial_x^\alpha g^{jk}(x)|
+\<x\>^{-m/2}|\partial_x^\alpha A_j(x)|
+\<x\>^{-m}|\partial^\alpha_x V(x)|&\le C_\alpha \<x\>^{-|\alpha|}.
\end{align}
\end{assumption}

When $2\le m<4$ and $g^{jk}=\delta_{jk}$, we also suppose the following: 
\begin{assumption}
\label{assumption_B}
Let $(B_{jk}(x))=(\partial_jA_k(x) - \partial_kA_j(x))$ be the magnetic field associated with $A$. Then there exists $\mu>0$ such that for any $\alpha\in \Z^d_+$ with $|\alpha|\ge1$, 
$$
|\partial_x^\alpha B_{jk}(x)| \le C_\alpha \<x\>^{m/2-2-\mu}.
$$
\end{assumption}

Here $\<x\>$ stands for $\sqrt{1+|x|^2}$. Under Assumption \ref{assumption_A}, the operator $\wtilde P$, with the domain $C_0^\infty(\R^d)$, is symmetric in $L^2(\R^d)$. Let $P$ be any one of its self-adjoint extensions. Then we consider the time-dependent Schr\"odinger equation
\begin{align}
\label{equation_1}
i\partial_t u=Pu,\ t\in\R;\quad u|_{t=0}=u_0\in L^2(\R^d).
\end{align}
The solution is given by 
$
u(t)=e^{-itP}u_0
$ by Stone's theorem, where $e^{-itP}$ denotes a unitary propagator on $L^2(\R^d)$ generated by $P$.

In this paper we are interested in the (local-in-time) \emph{Strichartz estimates} of the forms:\begin{align}
\label{Strichartz_estimates}
\norm{e^{-itP}u_0}_{L^p_TL^q} \le C_T\norm{u_0}_{\B^\gamma},
\end{align}
where $\gamma\ge0$, $L^p_TL^q:=L^p([-T,T];L^q(\R^d))$ and $(p,q)$ satisfies the following \emph{admissible condition}
\begin{align}
\label{admissible_pair}
2\le p,q\le\infty,\quad 
2/p=d(1/2+1/q),\quad
(d,p,q)\neq (2,2,\infty).
\end{align}
$\B^{s}=\B^s(\R^d)$ are weighted Sobolev spaces defined by
\begin{align}
\label{Sobolev_space}
\B^s(\R^d):=\{f\in \S'(\R^d);\ (1-\Delta+|x|^m)^{s/2}f\in L^2(\R^d)\},\quad s\in\R,
\end{align}
with the norm $\norm{f}_{\B^s}:=\norm{(1-\Delta+|x|^m)^{s/2}f}_{L^2(\R^d)}$. These spaces are characterized as follows:
\begin{align*}
\B^s(\R^d)&=\{f\in \S'(\R^d);\ \<x\>^{ms/2}f\in L^2(\R^d),\ \<D\>^sf\in  L^2(\R^d)\},\quad s\ge0,\\
\B^{s}(\R^d)&=\B^{-s}(\R^d)',\quad s<0. 
\end{align*}
Strichartz estimates can be regarded as $L^p$-type smoothing properties of Schr\"odinger equations and have been widely used in the study of nonlinear Schr\"odinger equations (see, \emph{e.g.}, \cite{Cazenave}). 

When $g^{jk}=\delta_{jk}$ and $A\equiv0$, the following has been proved by Yajima-Zhang \cite{Yajima_Zhang_2}:

\begin{theorem}[Theorem 1.3 of \cite{Yajima_Zhang_2}]						
\label{theorem_Yajima_Zhang}
Let $d\ge1$ and $H=-\Delta/2+V$ satisfy Assumption \ref{assumption_A} with some $m\ge2$ and
\begin{align}
\label{assumption_V}
V(x)\ge C\<x\>^m\quad\text{for}\ |x|\ge R, 
\end{align}
with some $R,C>0$. Then, for any $\ep,T>0$ and $(p,q)$ satisfying \eqref{admissible_pair} there exists $C_{T,\ep}>0$ such that
\begin{align}
\label{theorem_Yajima_Zhang_1}
\norm{e^{-itH}u_0}_{L^p_TL^q} \le C_{T,\ep}\norm{\<H\>^{\frac1p\left(\frac12-\frac1m\right)+\ep}u_0}_{L^2}.
\end{align}
\end{theorem}

Note that, under the conditions \eqref{assumption_A_1} and \eqref{assumption_V}, we have the norm equivalence:
$$
C_1\norm{u_0}_{\B^{\frac1p\left(1-\frac2m\right)+2\ep}}\le\norm{\<H\>^{\frac1p\left(\frac12-\frac1m\right)+\ep}u_0}_{L^2}\le C_2 \norm{u_0}_{\B^{\frac1p\left(1-\frac2m\right)+2\ep}}
$$
with some $C_1,C_2>0$ independent of $u_0$. (see Remark \ref{remark_2}.) 

In this paper we extend this theorem to the variable coefficient case under some geometric conditions on the Hamilton flow generated by the kinetic energy. Furthermore, we remove the additional loss $\<H\>^\ep$ for  the flat case. 

To state our main results, we introduce some notation on the classical system. Let $k(x,\xi)$ be the classical kinetic energy function:
$$
k(x,\xi)=\frac12g^{jl}(x)\xi_j\xi_l,\quad x,\xi\in\R^d.
$$
We denote by $(y_0(t,x,\xi),\eta_0(t,x,\xi))$ the Hamilton equation generated by $k$, \emph{i.e.},  the solution to
$$
\frac{d}{dt}y_0(t)=\frac{\partial k}{\partial \xi}(y_0(t),\eta_0(t)),\ \frac{d}{dt}\eta_0(t)=-\frac{\partial k}{\partial \xi}(y_0(t),\eta_0(t)) 
$$
with the initial condition $(y_0(0,x,\xi),\eta_0(0,x,\xi))=(x,\xi)$. Note that the Hamiltonian vector field $H_{k}=\partial_\xi k\cdot\partial_x-\partial_x k\cdot\partial_\xi$ is complete on $\R^{2d}$ since $(g^{jk})_{j,k}$ is uniformly elliptic. $(y_0(t),\eta_0(t))$ thus exists for all $t \in \R$. To control its asymptotic behavior, we then impose the following conditions:

\begin{assumption}
\label{assumption_C}
(1) (Nontrapping condition) For any $(x,\xi) \in \R^{2d}$ with $\xi \neq 0$, 
$$
|y_0(t,x,\xi)| \to +\infty\quad\text{as}\ t \to \pm \infty.
$$
(2) (Convexity near infinity) There exists $f\in C^\infty(\R^d)$ satisfying $f\ge1$ and $\lim_{|x|\to+\infty}f(x)=+\infty$ such that $\partial_x^\alpha f\in L^\infty(\R^d)$ for any $|\alpha|\ge2$ and that, for some constants $c,R>0$, 
$$
H_k (H_kf)(x,\xi)\ge ck(x,\xi)
$$
for all $x,\xi\in\R^d$ with $f(x)\ge R$. 
\end{assumption}

\begin{remark}
\label{remark_assumption_B}
 It is easy to see that if the quantity
$$
\sup_{|\alpha|\le2}\<x\>^{|\alpha|}|\partial_x^\alpha (g^{jk}(x)-\delta_{jk})|
$$
is sufficiently small, then $\partial_t^2(|y_0(t)|^2)\gtrsim |\xi|^2$ and hence Assumption \ref{assumption_C} (1) holds. Under the same condition, Assumption \ref{assumption_C} (2) also holds with $f(x)=1+|x|^2$. Moreover, if $g^{jk}(x)=(1+a_1\sin(a_2\log r))\delta_{jk}$ for $a_1\in\R,a_2>0$ with $a_1^2(1+a_2^2)<1$ and for $r=|x|\gg1$, then Assumption \ref{assumption_C} (2) holds with $f(r)=(\int_0^r(1+a_1\sin(a_2\log t))^{-1}dt)^2$. For more examples, we refer to \cite[Section 2]{Doi}. 
\end{remark}

We now state main results.

\begin{theorem}					
\label{theorem_1}
Let $d\ge2$, $m\ge2$ and $P$ satisfy Assumptions \ref{assumption_A} and \ref{assumption_C}. Then, for any $T,\ep>0$ and $(p,q)$ satisfying \eqref{admissible_pair}, there exists $C_{T,\ep}>0$ such that for any $u_0\in C_0^\infty(\R^d)$, 
\begin{align}
\label{theorem_1_1}
\norm{e^{-itP}u_0}_{L^p_TL^q} 
\le C_{T,\ep}\norm{u_0}_{\B^{\frac1p\left(1-\frac2m\right)+\ep}}.
\end{align}
\end{theorem}

For the flat case, we can remove the additional $\ep$-loss as follows.
\begin{theorem}
\label{theorem_2}
Let $d\ge3$, $m\ge2$ and $H=\frac12(D-A(x))^2+V(x)$ satisfy Assumption \ref{assumption_A}. When $2\le m<4$ we also suppose Assumption \ref{assumption_B}. 
Then, for any $T>0$, $(p,q)$ satisfying \eqref{admissible_pair} and for any $u_0\in C_0^\infty(\R^d)$
\begin{align}
\label{theorem_1_2}
\norm{e^{-itH}u_0}_{L^p_TL^q} 
\le C_{T}\norm{u_0}_{\B^{\frac1p\left(1-\frac2m\right)}}.
\end{align}
\end{theorem}

\begin{corollary}
\label{remark_1}
In Theorem \ref{theorem_2}, if we, in addition, assume \eqref{assumption_V} and $q<\infty$ then \eqref{theorem_1_2} holds for any dimension.
\end{corollary}

Several remarks are in order: 

\begin{remark}
It is known that the condition $V\ge -C\<x\>^2$ with some  $C>0$ is almost optimal for the essential self-adjointness of $\wtilde P$. More precisely, it is known (see, \emph{e.g.}, \cite[Chaper 13, Section 6]{Dunford_Schwartz}) that $-\Delta-|x|^4$ is not essentially self-adjoint on $C_0^\infty(\R)$. This is due to the fact that the corresponding classical trajectories blow up in a finite time in the sense that it can reach an infinite speed. However, it was shown in \cite{Iwatsuka} that $\wtilde P$ can be essentially self-adjoint even if $V$ blows up in the negative direction such as $V\le -C\<x\>^m$ with $m>2$, if strongly divergent magnetic fields are present near infinity. More precisely, we define
$$
|B(x)|:=\Big(\sum_{j<k}B_{jk}(x)^2\Big)^{1/2}.
$$
Then if we assume, in addition Assumption \ref{assumption_A}, that $V(x)+|B(x)|\ge -C\<x\>^2$ with some $C>0$ then $\wtilde P$ is essentially self-adjoint on $C_0^\infty(\R^d)$. In this case, since $C_0^\infty(\R^d)$ is a dense subset of $\B^s(\R^d)$ if $s\ge0$, by the density argument, Theorem \ref{theorem_1} (resp. Theorem \ref{theorem_2}) holds for any $u_0 \in\B^s(\R^d)$ with $s=\frac1p\left(1-\frac2m\right)+\ep$ (resp. $s=\frac1p\left(1-\frac2m\right)$). 
\end{remark}

\begin{remark}
\label{remark_2}
Suppose that $V$ satisfies \eqref{assumption_V}. Then we can assume $P\ge1$ without loss of generality and $P$ hence is uniformly elliptic in the sense that
$$
p(x,\xi)\approx |\xi|^2+\<x\>^m,
$$
where $p$ is the full hamiltonian associated to $P$ (modulo lower order term), \emph{i.e.}, 
$$
p(x,\xi)=\frac12g^{jk}(x)(\xi_j-A_j(x))(\xi_k-A_k(x))+V(x).
$$
By the standard parametrix construction for $P$, we see that, for any $1<q<\infty$ and $s\ge0$
\begin{align*}
\norm{P^{s/2}v}_{L^q}+\norm{v}_{L^q}
\approx \norm{\<D\>^sv}_{L^q}+\norm{\<x\>^{ms/2}v}_{L^q},\quad v\in C_0^\infty(\R^d),
\end{align*}
(see, \emph{e.g.}, \cite[Lemma 2.4]{Yajima_Zhang_2}). In particular, $\norm{\<P\>^{s/2}v}_{L^2}\approx \norm{v}_{\B^s}$. 
Therefore, our result is a generalization and improvement of Theorem \ref{theorem_Yajima_Zhang}. 
\end{remark}

\begin{remark}
 If $m=2$ then Theorem \ref{theorem_2} implies Strichartz estimates without loss of derivatives:
$$
\norm{e^{-itH}u_0}_{L^p_TL^q}\le C_T\norm{u_0}_{L^2(\R^d)},
$$
which have already been proved by Yajima \cite{Yajima2} (see also Fujiwara \cite{Fujiwara} for the case $A\equiv0$). We also note that, for the Harmonic oscillator $H=-\frac12\Delta+\frac{|x|^2}{2}$, the distribution kernel of the propagator $e^{-itH}$ is given by the Mehlerfs formula:
$$
E(t,x,y)=\frac{1}{i^m(2\pi i\sin t)^{n/2})}\exp\left(-\frac{(|x|^2+|y|^2)\cos t-2xy}{2i\sin t}\right)
$$
for $m\pi<t<(m+1)\pi$ and $m\in \Z$, and thus satisfies the dispersive estimate:
\begin{align}
\label{dispersive}
\norm{e^{-itH}u_0}_{L^\infty} \le C|t|^{-d/2}\norm{u_0}_{L^1}
\end{align}
if $0<|t|<\pi$. 
Then Strichartz estimates are direct consequences of this inequality and the $TT^*$-argument due to \cite{Ginibre_Velo} (see \cite{Keel_Tao} for the endpoint). 
\end{remark}

\begin{remark}
The additional $\ep$-loss in \eqref{theorem_1_1} is only due to the use of the following local smoothing effect under Assumption \ref{assumption_C}:
$$
\norm{\<x\>^{-1/2-\ep}E_{1/m}e^{-itP}u_0}_{L^2_TL^2} \le C_{T,\ep}\norm{u_0}_{L^2},\quad\ep>0,
$$
where $E_{s}$ is a pseudodifferential operator with the symbol $(k_A(x,\xi)+\<x\>^m)^{s/2}$, where $k_A(x,\xi)=\frac12(\xi_j-A_j(x))g^{jk}(x)(\xi_k-A_k(x))$. It is well known that this estimate does not holds when $\ep=0$ even for $P=-\frac12\Delta+\<x\>^m$ (see \cite{RZ}). Also note that we do not use such a local smoothing effect to prove Theorem \ref{theorem_2}. 
\end{remark}

\begin{remark}
\label{remark_trapping}
Let $d\ge3$. Without Assumption \ref{assumption_C} we only have the following estimates
\begin{align}
\label{Strichartz_trapping}
\norm{e^{-itP}u_0}_{L^p_TL^q}\le C_T\norm{u_0}_{\B^{\frac1p}},
\end{align}
which are weaker than \eqref{theorem_1_1} and basically same as that for the case on the Laplace-Beltrami operator on compact manifolds $(M,g)$ without boundary (cf. \cite{BGT}) in which it was shown that $e^{it\Delta_g}$ obeys \eqref{Strichartz_trapping} with $\norm{u_0}_{\B^{\frac1p}}$ replaced by $\norm{(1-\Delta_g)^{1/(2p)}u_0}_{L^2}$. We will give the proof of \eqref{Strichartz_trapping} in the end of the proof of Theorem \ref{theorem_1}. 
\end{remark}

Global-in-time Strichartz estimates, that is \eqref{Strichartz_estimates} with $T=\infty$ and $\gamma=0$, for the free propagator $e^{it\Delta/2}$ were first proved by Strichartz \cite{Strichartz} for a restricted pair of $(p,q)$ with $p=q=2(d+2)/d$, and have been generalized for $(p,q)$ satisfying \eqref{admissible_pair} and $p\neq2$ by \cite{Ginibre_Velo}. The endpoint estimate $(p,q)=(2,2d/(d-2))$ for $d\ge3$ was obtained by \cite{Keel_Tao}. 

Furthermore, Strichartz estimates for Schr\"odinger equations have been extensively studied by many authors for both of potential and metric perturbation cases, separately. 

For Schr\"odinger operators with potentials satisfying Assumption \ref{assumption_A} with $m\le2$, it was shown by  \cite{Fujiwara,Yajima2} that $e^{-itH}\varphi$ satisfies (short-time) dispersive estimate \eqref{dispersive} for sufficiently small $t\neq0$ and hence local-in-time Strichartz estimates without loss of derivatives. For the case with singular electric potentials, we refer to \cite{Yajima1}. We mention that global-in-time dispersive and Strichartz estimates for scattering states have been also studied under suitable decaying conditions on potentials and assumptions for zero energy; see \cite{JSS,Yajima5,Schlag1} for dispersive estimates and \cite{Rodnianski_Schlag,BPSS,D'AFVV} for Strichartz estimates, and reference therein. We also mention that there is no result on sharp global-in-time dispersive estimates for (generic) magnetic Schr\"odinger operators, though \cite{FFFM} has recently proved dispersive estimates for the Aharonov-Bohm effect in $\R^2$. 

On the other hand, it is important to study the influence of underlying classical dynamics on the behavior of solutions to linear and nonlinear partial differential equations. From this geometric viewpoint, local-in-time Strichartz estimates for metric perturbations (or, more generally, on manifolds) have recently been investigated by many authors under several conditions on the geometry; see, \emph{e.g.}, \cite{Staffilani_Tataru,BGT,Robbiano_Zuily,HTW,Bouclet_Tzvetkov_1,Bouclet,BGH,Mizutani1} and reference therein. We mention that there are also several works on global-in-time Strichartz estimates in the case of long-range perturbations of the flat Laplacian on $\R^d$ (\cite{Bouclet_Tzvetkov_2,Tataru,MMT}). 

The main purpose of  this paper is to handle the mixed case, that is the case for metric perturbations with unbounded electromagnetic potentials. In the previous works \cite{Mizutani2,Mizutani3}, we proved the same local-in-time Strichartz estimates as in the free case under Assumptions \ref{assumption_A} and \ref{assumption_C} with $m<2$ and the following long-range condition
$$
|\partial_x^\alpha(g^{jk}(x)-\delta_{jk})|\le C_\alpha\<x\>^{-\mu-|\alpha|},\quad \mu>0.
$$
This paper is a natural continuation of this work, and the results in the series of works can be regarded as a generalization and unification of many of known local-in-time Strichartz estimates for Schr\"odinger equations with both of metric and unbounded potential perturbations, at least under the nontrapping condition. 

\subsection{Notations}
Throughout the paper we use the following notations: 
We write $L^q=L^q(\R^d)$ if there is no confusion. 
$W^{s,q}=W^{s,q}(\R^d)$ is the Sobolev space with the norm $\norm{f}_{W^{s,q}}=\norm{\<D\>^sf}_{L^q}$. 
For Banach spaces $X$ and $Y$, $\norm{\cdot}_{X\to Y}$ denotes the operator norm from $X$ to $Y$. 
For constants $A,B\ge0$, $A\lesssim B$ means that there exists some universal constant $C>0$ such that $A \le CB$. 
$A\approx B$ means $A\lesssim B$ and $B\lesssim A$. 

We always use the letter $P$ (resp. $H$) to denote variable coefficient (resp. flat)  Schr\"odinger operators. For $h\in (0,1]$, we consider $P^h:=h^2P$ as a semiclassical Schr\"odinger operator with $h$-dependent potentials $h^2V$ and $hA_j$. We set two corresponding $h$-dependent symbols $p^h$ and $p_1^h$ defined by
\begin{equation}
\begin{aligned}
\label{symbols_3}
p^h(x,\xi)&=\frac12g^{jk}(x)(\xi_j-hA_j(x))(\xi_k-hA_k(x))+h^2V(x),\\
p_{1}^h(x,\xi)&=-\frac i2\frac{\partial g^{jk}}{\partial x_j}(x)(\xi_k-hA_k(x))-\frac{ih}{2}g^{jk}(x)\frac{\partial A_k}{\partial x_j}(x).
\end{aligned}
\end{equation}
It is easy to see that $P^h=p^h(x,hD)+hp_{1}^h(x,hD)$ and that Assumption \ref{assumption_A} implies
\begin{equation}
\begin{aligned}
\label{symbols_2}
|\dderiv{x}{\xi}{\alpha}{\beta}p^h(x,\xi)|
&\le C_{\alpha\beta}\<x\>^{-|\alpha|}\<\xi\>^{-|\beta|}(|\xi|^2+h^2\<x\>^m),\\
|\dderiv{x}{\xi}{\alpha}{\beta}p_1^h(x,\xi)|
&\le C_{\alpha\beta}\<x\>^{-1-|\alpha|}\<\xi\>^{-|\beta|}(|\xi|+h\<x\>^{m/2}).
\end{aligned}
\end{equation}

\subsection{Strategy of the proof}
Before starting the details of the proof of main theorems, we here explain the basic idea. For simplicity we may assume $d\ge3$. 
The strategy is based on microlocal techniques and basically follows the same general lines in \cite{JSS,Staffilani_Tataru,BGT} and \cite[Sections 5 and 6]{Bouclet_Tzvetkov_1}. 
We however note that, since the potential $V$ is not bounded below, the Littlewood-Paley theory in terms of the spectral multiplier $f(P)$ does not work well. For example, the Littlewood-Paley  estimates of the forms
\begin{align}
\label{strategy_0}
\norm{v}_{L^q}\lesssim \norm{v}_{L^2}+\Big(\sum_{j=0}^\infty\norm{f(2^{-2j}P)v}_{L^q}^2\Big)^{1/2},\quad f\in C_0^\infty(\R^d\setminus\{0\}),
\end{align}
seem to be false except for the trivial case $q=2$. In particular, it is difficult to apply the localization technique due to \cite{BGT} directly. 

To overcome this difficulty, we introduce a partition of unity $\psi_0+\psi_1=1$ on the phase space $T^*\R^d\cong \R^d_x\times\R^d_\xi$ with symbols $\psi_j\in S(1,dx^2/\<x\>^2+d\xi^2/\<\xi\>^2)$ (see Section 2 for the definition of $S(1,dx^2/\<x\>^2+d\xi^2/\<\xi\>^2)$) supported in the following ``high frequency" and ``low frequency" regions, respectively:
$$
\supp \psi_0\subset\{\<x\>^m< \ep|\xi|^2\},\quad
\sup \psi_1\subset\{\<x\>^m>\ep|\xi|^2/2\},
$$ 
where $\ep>0$ is a sufficiently small constant. 
Let $c>1$ and consider a $c$-adic partition of unity: 
$
\theta_0,\theta\in C_0^\infty(\R)$, $0\le \theta_0, \theta \le 1$, $
\supp \theta\subset(1/c,c)$, $\theta_0(z)+\sum_{j=0}^\infty\theta(c^{-j}z)=1$. 
We use this decomposition with $z=\xi,c=2$ in the high frequency region and with $z=x,c=2^{2/m}$ in the low frequency region, respectively. More precisely, setting $h=2^j$ and using support properties
\begin{align*}
\supp\theta(\xi)\psi_0(x,\xi/h)&\subset\{\<x\>\lesssim h^{-2/m},\ |\xi|\approx1\},\\
\supp\theta(h^{2/m}x)\psi_1(x,\xi/h)&\subset\{\<x\>\approx h^{-2/m},\ |\xi|\lesssim1\}, 
\end{align*}
 we will prove the following Littlewood-Paley type estimates:
$$
\norm{v}_{L^q}
\lesssim 
\norm{v}_{L^2}+
\sum_{k=0,1}\Big(\sum_{j=0}^\infty\norm{\Psi_k^h(x,hD)v}_{L^q}^2\Big)^{1/2},\quad q\in[2,\infty),
$$
where $\Psi_0^h(x,\xi)=\theta(\xi)\psi_0(x,\xi/h)$ up to $O(h^\infty)$ and $\Psi_1^h(x,\xi)=\theta(h^{2/m}x)\psi_1(x,\xi/h)$. Using support properties of $\Psi_k^h$, we obtain the following bounds of the commutators:
\begin{align}
\label{strategy_1}
[P,\Psi_0^h(x,hD)]=O(\<x\>^{-1}h^{-1}),\quad [P,\Psi_1^h(x,hD)]=O(h^{-1+2/m}). 
\end{align}
These terms can be controlled by the local smoothing effect, and the proof of Theorem \ref{theorem_1} thus is reduced to that of the estimates for the localized propagators $\Psi_k^h(x,hD)e^{-itP}$. 

Then we show that, for any symbol $\chi^h\in S(1,dx^2/\<x\>^2+d\xi^2/\<\xi\>^2)$ supported in $\{(x,\xi);\ \<x\>\lesssim h^{-2/m},\ |\xi|\lesssim1\}$, $\chi^h(x,hD)e^{-itP}\chi^h(x,hD)^*$ satisfies the following dispersive estimate:
\begin{align}
\label{strategy_2}
\norm{\chi^h(x,hD)e^{-itP}\chi^h(x,hD)^*}_{L^1\to L^\infty}
\lesssim |t|^{-d/2},\quad |t|\ll hR,\ h\in(0,1],
\end{align}
where $R=\inf |\pi_x(\supp \chi^h)|+1$ for general cases and $R=h^{-2/m}$ for the flat case. 
This estimate can be verified by a slightly refinement of the standard semiclassical parametrix construction. 
Namely, after rescaling $t\mapsto th$ and putting $P^h=h^2P$, we construct the following semiclassical WKB parametrix for $e^{-itP^h/h}\chi^h(x,hD)^*$ up to a time scale of order $O(R)$:
$$
e^{-itP^h/h}\chi^h(x,hD)^*=J_{S^h}(a^h)+O_{L^2\to L^2}(h^\infty),\quad |t|\ll R,
$$
where $J_{S^h}(a^h)$ is a time-dependent semiclassical Fourier integral operator:
$$
J_{S^h}(a^h)u_0(x)=(2\pi h)^{-d}\int e^{i(S^h(t,x,\xi)-y\cdot\xi)/h}a^h(t,x,\xi)u_0(y)dyd\xi.
$$
Here the phase function $S^h(t,x,\xi)$ solves the Hamilton-Jacobi equation associated to $p^h$ on a neighborhood of $\supp \chi^h$ and satisfies
$$
S^h(t,x,\xi)=x\cdot\xi-tp^h(x,\xi)+O(R^{-1}|t|^2),\quad |t|\ll R.
$$
The amplitude $a^h$ approximately solves the transport equation generated by the vector field $\partial_\xi p^h(x,\partial_xS^h)$ and belongs to $S(1,dx^2/\<x\>^2+d\xi^2/\<\xi\>^2)$ with uniform bounds in $h$ and $t$. 
Furthermore, the stationary phase method yields
$$
\norm{J_{S^h}(a^h)}_{L^1\to L^\infty}\lesssim \min(h^{-d},|th|^{-d/2}),\quad |t|\ll R,
$$
which, together with suitable estimates for the error term, implies \eqref{strategy_2}. 

Once we obtain the dispersive estimate, a standard technique due to Ginibre-Velo \cite{Ginibre_Velo} (see also Staffilani-Tataru \cite{Staffilani_Tataru}), together with the $TT^*$-argument \cite{Ginibre_Velo} (see also Keel-Tao \cite{Keel_Tao}) and \eqref{strategy_1}, imply the following endpoint Strichartz estimate with an inhomogeneous error term:
\begin{equation}
\begin{aligned}
\label{strategy_3}
&\norm{\Psi_k^h(x,hD)e^{-itP}u_0}_{L^2_TL^{\frac{2d}{d-2}}}\\
&\lesssim
\norm{\Psi_k^h(x,hD)u_0}_{L^2}
+\norm{\<x\>^{-1/2}h^{-1/2}\Psi_k^h(x,hD)e^{-itP}u_0}_{L^2_TL^2}.
\end{aligned}
\end{equation}
Combining with an almost orthogonality of $\theta(c^{-j}\cdot)$, \emph{i.e.}, $\sum\limits_j \norm{\theta(c^{-j}z)v}_{L^2}^2\approx \norm{v}_{L^2}^2$, we have
$$
\norm{e^{-itP}u_0}_{L^2_TL^{\frac{2d}{d-2}}}
\lesssim
\norm{u_0}_{L^2}
+\norm{\<x\>^{-1/2-m\ep/2}E_{1/2+\ep}e^{-itP}u_0}_{L^2_TL^2},\quad \ep\ge0.
$$
Under Assumption \ref{assumption_C}, we then use the local smoothing effect due to \cite{RZ}:
$$
\norm{\<x\>^{-1/2-\nu}E_{1/m+s}e^{-itP}u_0}_{L^2_TL^2}\le C_{T,\nu}\norm{E_su_0}_{L^2},\quad \nu>0,
$$
 and obtain $\norm{e^{-itP}u_0}_{L^2_TL^{\frac{2d}{d-2}}}\lesssim \norm{E_{1/2-1/m+\ep}u_0}_{L^2}$ if $\ep>0$. Finally, Theorem \ref{theorem_1} is verified by interpolation with the trivial $L^\infty_TL^2$-bound. 
 
 In the flat case, the last term of \eqref{strategy_3} can be replaced by 
$$
\norm{h^{-1/2+1/m}\Psi_k^h(x,hD)e^{-itH}u_0}_{L^2_TL^2},
$$ 
 due to the fact that we obtain \eqref{strategy_2} with $R=h^{-2/m}$. Therefore, using the following energy estimate
$$
\norm{E_se^{-itH}u_0}_{L^2}\le Ce^{C|t|}\norm{E_su_0}_{L^2}
$$
instead of the local smoothing effect, one can remove the $\ep$-loss and obtain Theorem \ref{theorem_2}. 

The paper is organized as follows: We first record some known results on the semiclassical pseudodifferential calculus and prove the above Littlewood-Paley estimates in Section \ref{Preliminaries}. Section \ref{Preliminaries} also discusses local smoothing effect and energy estimates as above. In Section 3, we construct the WKB parametrix and prove dispersive estimates \eqref{strategy_2}. Proofs of Theorems \ref{theorem_1} and \ref{theorem_2} are given in Section 4. We finally prove Corollary \ref{remark_1} in Appendix A with a simpler proof than that of main theorems.


\section{Preliminaries}
\label{Preliminaries}
In this section, we record some known results on the semiclassical pseudodifferential calculus and the Littlewood-Paley theory. This section also discuss local smoothing effects for the propagator $e^{-itP}$ under Assumption \ref{assumption_C}. 

First of all we collect basic properties of the semiclassical pseudodifferential operator ($h$-$\Psi$DO for short). We omit proofs and refer to \cite{Robert2,Martinez} for the details. 
Set a metric on the phase space $T^*\R^d\cong \R^{2d}$ defined by 
$
g=dx^2/\<x\>^2+d\xi^2/\<\xi\>^2
$.  
For a $g$-continuous weight function $m(x,\xi)$, we use H\"ormander's symbol class $S(m,g)$ (see \cite[Chaper 18]{Hormander} and \cite[Section 3]{Doi} for the details), which is the space of smooth functions on $\R^{2d}$ satisfying
$$
|\dderiv{x}{\xi}{\alpha}{\beta}a(x,\xi)|\le C_{\alpha\beta}m(x,\xi)\<x\>^{-|\alpha|}\<\xi\>^{-|\beta|}.
$$
To a symbol $a \in C^\infty(\R^{2d})$ and $h\in(0,1]$, we associate the $h$-$\Psi$DO $a(x,hD)$ defined by
$$
a(x,hD)f(x)=(2\pi h)^{-d}\int e^{i(x-y)\cdot\xi/h}a(x,\xi)f(y)dyd\xi,\quad f\in \S(\R^d),
$$
where $\S(\R^d)$ is the Schwartz class. For a $h$-$\Psi$DO $A$, we denote its symbol by $\Sym(A)$, \emph{i.e.}, $A=a(x,hD)$ if $a=\Sym(A)$. It is known as the Calder\'on-Vaillancourt theorem that for any symbol $a\in C^\infty(\R^{2d})$ satisfying
$
|\dderiv{x}{\xi}{\alpha}{\beta}a(x,\xi)|\le C_{\alpha\beta},
$
$a(x,hD)$ is extended to a bounded operator on $L^2(\R^d)$ with a uniform bound in $h\in(0,1]$.
Moreover, if  
$
|\dderiv{x}{\xi}{\alpha}{\beta}a(x,\xi)|\le C_{\alpha\beta}\<\xi\>^{-\gamma}
$
with some $\gamma >d$, then $a(x,hD)$ is extended to a bounded operator from $L^q$ to $L^r$ with bounds
\begin{align}
\label{pdo_1}
\norm{a(x,hD)}_{L^q \to L^r} \le C_{qr}h^{-d(1/q-1/r)},\quad 1\le q\le r \le\infty,
\end{align}
where $C_{qr}>0$ is independent of $h\in(0,1]$. These bounds follow from the Schur lemma and the Riez-Thorin interpolation theorem (see, \emph{e.g.},  \cite[Proposition 2.4]{Bouclet_Tzvetkov_1}). For two symbols $a\in S(m_1,g)$ and $b\in S(m_2,g)$, $a(x,hD)b(x,hD)$ is also a $h$-$\Psi$DO with the symbol $a\sharp b(x,\xi)=e^{ihD_\eta D_z}a(x,\eta)b(z,\xi)|_{z=x,\eta=\xi}\in S(m_1m_2,g)$, which has the expansion
\begin{align}
\label{pdo_2}
a\sharp b-\sum_{|\alpha|<N}\frac{h^{|\alpha|}}{i^{|\alpha|}\alpha!}\partial_\xi^\alpha a\cdot \partial_x^\alpha b\in S(h^N\<x\>^{-N}\<\xi\>^{-N}m_1m_2,g).
\end{align}
In particular, we have $\Sym([a(x,hD),b(x,hD)])-\frac hi\{a,b\}\in S(h^2\<x\>^{-2}\<\xi\>^{-2},g)$, where $\{a,b\}=\partial_\xi a\cdot\partial_xb-\partial_xa\cdot\partial_\xi b$ is the Poisson bracket. 
The symbol of the adjoint $a(x,hD)^*$ is given by $a^*(x,\xi)=e^{ihD_\eta D_z}a(z,\eta)|_{z=x,\eta=\xi}\in S(m_1,g)$ which has the expansion
\begin{align}
\label{pdo_3}
a^*-\sum_{|\alpha|<N}\frac{h^{|\alpha|}}{i^{|\alpha|}\alpha!}\partial_\xi^\alpha\partial_x^\alpha a\in S(h^N\<x\>^{-N}\<\xi\>^{-N}m_1,g).
\end{align}
We also often use the following which is a direct consequence of \eqref{pdo_2}:

\begin{lemma}
\label{lemma_pdo_1}
Let $a\in S(m_1,g)$ and $b\in S(m_2,g)$. If $b\equiv1$ on $\supp a$, then for any $N\ge0$, 
$$
a(x,hD)=a(x,hD)b(x,hD)+h^Nr_N(x,hD)=b(x,hD)a(x,hD)+h^N\wtilde r_N(x,hD)
$$
with some $r_N,\wtilde r_N\in S(\<x\>^{-N}\<\xi\>^{-N}m_1m_2,g)$. 
\end{lemma}


\subsection{Littlewood-Paley estimates}
\label{LP}
We here prove Littlewood-Paley type estimates, which will be used to reduce the proof of the estimates \eqref{theorem_1_1} to that of semiclassical Strichartz estimates. Here and in what follows, the summation over $h$, $\sum\limits_h$, means that $h$ takes all negative powers of 2 as values, \emph{i.e.}, $\sum\limits_h:=\sum\limits_{h=2^{-j},\,j\ge0}$. 

\begin{proposition}[Littlewood-Paley estimates]
\label{proposition_LP_1}
For $h\in(0,1]$, there exist two symbols $\Psi^h_0$ and $\Psi_1^h$ such that the following statements are satisfied:\\
\emph{(1)} \emph{(}Symbol estimates\emph{)} $\{\Psi^h_k\}_{h\in(0,1]}$ are bounded in $S(1,h^{4/m}dx^2+d\xi^2/\<\xi\>^2)$, i.e.,
$$
|\dderiv{x}{\xi}{\alpha}{\beta}\Psi_k^h(x,\xi)|\le C_{\alpha\beta}h^{(2/m)|\alpha|}\<\xi\>^{-|\beta|},\quad k=0,1,
$$
uniformly with respect to $h\in(0,1]$. \\
\emph{(2} \emph(Support properties\emph) There exist $C>0$ and $C_\ep>0$, independent of $x,\xi\in\R^d$ and $h\in (0,1]$, such that
\begin{align}
\label{proposition_LP_1_1}
\supp \Psi_0^h&\subset\{(x,\xi);\ h^2\<x\>^m\le C_\ep,\ C^{-1}\le |\xi|^2\le C\},\\
\label{proposition_LP_1_2}
 \supp \Psi_1^h&\subset\{(x,\xi);\ C_\ep^{-1}\le h^2\<x\>^m\le C_\ep,\ |\xi|^2\le C\}.
\end{align}
\emph{(3)} \emph(Littlewood-Paley estimates\emph) For any $q\in [2,\infty)$,
\begin{align}
\label{proposition_LP_1_3}
\norm{v}_{L^q}
\le C_q
\norm{v}_{L^2}+
\sum_{k=0,1}\Big(\sum_{h}\norm{\Psi_k^h(x,hD)v}_{L^q}^2\Big)^{1/2}.
\end{align}
\end{proposition}

\begin{remark}
\label{remark_LP_1}
As we mentioned in the strategy of the proof, it seems to be difficult to obtain the Littlewood-Paley estimates \eqref{strategy_0} if $q\neq2$. We, however, note that if $V$ satisfies \eqref{assumption_V}, then one can prove \eqref{strategy_0} for $q\in[2,\infty)$ (see Appendix). 
\end{remark}

In order to prove Proposition \ref{proposition_LP_1}, we first construct the symbols $\Psi_k^h$. Let $\varphi\in C^\infty_0(\R)$ be such that
$
\supp \varphi\subset[-1,1],\ \varphi\equiv1\ \text{on}\ [-1/2,1/2]$ and $0\le\varphi\le1. 
$
For $\ep>0$, define smooth cut-off functions into high and low frequency regions by
$$
\psi_0(x,\xi)=\varphi\Big(\frac{\<x\>^{m/2}}{\ep^{1/2}|\xi|}\Big),\ \psi_1=1-\psi_0,
$$
respectively. 
It is easy to see that 
$$
\supp \psi_0\subset\{(x,\xi); \<x\>^m\le \ep|\xi|^2\},\ 
\supp\psi_1\subset\{(x,\xi); \<x\>^m\ge \ep|\xi|^2/2\}
$$ 
and that $\psi_0,\psi_1\in S(1,g)$ for each $\ep>0$, \emph{i.e.}, 
$$
|\dderiv{x}{\xi}{\alpha}{\beta}\psi_0(x,\xi)|+|\dderiv{x}{\xi}{\alpha}{\beta}\psi_0(x,\xi)|\le C_{\alpha\beta\ep}\<x\>^{-|\alpha|}\<\xi\>^{-|\beta|}.
$$
By Assumption \ref{assumption_A}, we can choose $\ep$, depending only on $\sup_x\<x\>^{-m/2}|A(x)|+\sup_x\<x\>^{-m}|V(x)|$, so small that 
$$
|\xi|^2/2\le p(x,\xi)\le |\xi|^2\quad\text{if}\quad
\<x\>^m\le \ep|\xi|^2. 
$$ 
In the sequel, we fix such a small constant $\ep$ and often suppress the dependence of constants $C_{\alpha\beta\ep}$ on $\ep$. 
\begin{lemma}							
\label{lemma_LP_2} 
For any $\theta\in C_0^\infty(\R^d)$ supported away from the origin and any $N>d$, there exists a bounded family $\{\Psi_0^h\}_{h\in(0,1]}\subset S(1,h^{4/m}dx^2+d\xi^2/\<\xi\>^2)$ satisfying \eqref{proposition_LP_1_1} such that
$$
\norm{\theta(hD)\psi_0(x,D)-\Psi_0^h(x,hD)}_{L^2\to L^q}\le C_{qN}h^{N-d(1/2-1/q)},\quad q\in[2,\infty),
$$
where $C_{qN}>0$ may be taken uniformly in $h\in(0,1]$. 
Moreover, if we set 
$$
\Psi_1^h(x,\xi):=\theta(h^{m/2}x)\psi_1(x,\xi/h),
$$ 
then $\{\Psi_1^h\}_{h\in(0,1]}$ is bounded in $S(1,h^{4/m}dx^2+d\xi^2/\<\xi\>^2)$ and satisfies the support property \eqref{proposition_LP_1_2}.
\end{lemma}

\begin{proof}
Choose $\wtilde \theta\in C_0^\infty(\R^d)$ so that $\wtilde \theta$ is supported away from the origin and that $\wtilde \theta\equiv1$ on $\supp \theta$. Then 
we learn by the expansion formula \eqref{pdo_2} (with $h=1$) that
$$
\theta(hD)\psi_0(x,D)
=\theta(hD)\wtilde \theta(hD) \psi_0(x,D)
=\theta(hD)\wtilde \psi_0^h(x,D)+\theta(hD)\wtilde r_N^h(x,D),
$$
where $\wtilde \psi_0^h\in S(1,g)$ is given by
$$
\wtilde \psi_0^h(x,\xi)=\sum_{|\alpha|<N}\frac{i^{-|\alpha|}}{\alpha!}\partial_\eta^\alpha\partial_z^\alpha \wtilde{\theta}(h\eta)\psi_0(z,\xi)\Big|_{\eta=\xi,z=x}
$$
such that the following support property holds:
\begin{equation}
\begin{aligned}
\label{proof_proposition_LP_1_1}
\supp \wtilde \psi_0^h(\cdot,\cdot/h)
&\subset\{(x,\xi);\ (x,\xi)\in \supp \psi_0(\cdot,\cdot/h),\ \xi\in\supp \theta\}\\
&\subset\{(x,\xi);\ h^2\<x\>^m\le \ep|\xi|^2\le C\ep,\ C^{-1}\le |\xi|\le C\},
\end{aligned}
\end{equation}
with some $C>0$. 
The remainder $\wtilde r_N^h$ belongs to $S(\<x\>^{-N}\<\xi\>^{-N},g)$ with uniform bounds in $h$, and hence $\<D\>^N \wtilde r_N^h(x,D)$ is bounded on $L^2$ uniformly in $h\in(0,1]$. Since $|\xi|\approx h^{-1}$ on $\supp \theta(h\cdot)$, we see that $|\xi|\ge \delta$ with some $\delta>0$ and hence
$$
\<\xi\>^{-N}\le C_N|\xi|^{-N}\le C_N h^{-N}\quad\text{on}\quad \supp \theta(h\cdot).
$$
Therefore, the estimate \eqref{pdo_1} with $(q,r)$ replaced by $(2,q)$ implies that
\begin{align*}
\norm{\theta(hD)\wtilde r_N^h(x,D)}_{L^2\to L^q}
&\le \norm{\theta(hD)\<D\>^{-N}}_{L^2 \to L^q}\norm{\<D\>^N\wtilde r_N^h(x,D)}_{L^2\to L^2}\\
&\le C_N h^N \norm{\theta(hD)(hD)^{-N}}_{L^2 \to L^q}\\
&\le C_{qN} h^{N-d(1/2-1/q)}. 
\end{align*}
For the main term, we have
\begin{align*}
|\dderiv{x}{\xi}{\alpha}{\beta}\wtilde \psi_0^h(x,\xi/h)|
&\le C_{\alpha\beta}\<x\>^{-|\alpha|}h^{-|\beta|}\<\xi/h\>^{-|\beta|}\\
&\le C_{\alpha\beta}\<x\>^{-|\alpha|}(h+|\xi|)^{-|\beta|}
\le C_{\alpha\beta}\<x\>^{-|\alpha|}\<|\xi|\>^{-|\beta|},
\end{align*}
where, in the last line, we have used the bound $|\xi|\ge C^{-1}$ on $\pi_\xi(\supp \wtilde\psi_0(\cdot,\cdot/h))\subset \supp\theta$. 
Therefore $\{\wtilde \psi_0^h(\cdot,\cdot/h)\}_{h\in(0,1]}$ is bounded in $S(1,g)$, while $\psi_0(x,\xi/h)$ may have singularities at $\xi=0$ as $h\to0$. 
In particular, $\wtilde \psi_0^h(x,D)$ can be regarded as a $h$-$\Psi$DO with the symbol $\wtilde \psi_0^h(\cdot,\cdot/h)$. \eqref{pdo_2} again implies that there exist bounded families $\{\Psi_0^h\}_{h\in(0,1]}\subset S(1,g)$ and $\{r_N^h\}_{h\in(0,1]}\subset S(\<x\>^{-N}\<\xi\>^{-N},g)$ such that
\begin{align*}
\theta(hD)\wtilde \psi_0^h(x,D)
=\Psi_0^h(x,hD)+h^N r_N^h(x,hD),
\end{align*}
where $\Psi_0^h$ is given explicitly by
\begin{align}
\label{proof_proposition_LP_1_2}
\Psi_0^h(x,\xi)=\sum_{|\alpha|<N}\frac{h^{|\alpha|}i^{-|\alpha|}}{\alpha!}\partial_\eta^\alpha\partial_z^\alpha \theta(\eta)\wtilde\psi_0^h(z,\xi/h)\Big|_{\eta=\xi,z=x},
\end{align}
and the remainder satisfies
$$\norm{r_N^h(x,hD)}_{L^2\to L^q}
\le C_{qN}h^{-d(1/2-1/q)},\quad 1\le q\le \infty.
$$
By virtue of \eqref{proof_proposition_LP_1_1} and \eqref{proof_proposition_LP_1_2}, we see that  
$$
\supp \Psi_0^h(x,\xi)
\subset\supp \theta \cap \supp \psi_0(\cdot,\cdot/h)
\subset\{h^2\<x\>^m\le C_0\ep,\ C_0^{-1}\le |\xi|\le C_0\}
$$
and that, for $|\alpha|\ge1$, 
$$
\supp \partial_x^\alpha\Psi_0^h
\subset\supp \theta \cap \supp \psi_0'(\cdot,\cdot/h)
\subset\{C_0^{-1}\ep\le h^2\<x\>^m\le C_0\ep,\ C_0^{-1}\le |\xi|\le C_0\},
$$
where $C_0>0$ is independent of $x,\xi,h$ and $\ep$. 
Therefore, $\Psi_0^h$ satisfies \eqref{proposition_LP_1_1} and 
$$
|\dderiv{x}{\xi}{\alpha}{\beta}\Psi_0^h(x,\xi)|
\le C_{\alpha\beta}\<x\>^{-|\alpha|}\<x\>^{-|\beta|}
\le C_{\alpha\beta}h^{2|\alpha|/m}\<\xi\>^{-|\beta|},\quad h\in(0,1]. 
$$

On the other hand, since $\dderiv{x}{\xi}{\alpha}{\beta}\psi_1$ are supported in 
$$
\supp \psi_0'\subset\{(x,\xi);\ \ep|\xi|^2/2<\<x\>^m\le \ep|\xi|^2\}
$$ for any $|\alpha+\beta|\ge1$, we learn 
$$
|\xi|\approx h^2\<x\>^m\approx\ep\quad\text{on}\quad\supp \theta(h^{2/m}x)\cap\supp\dderiv{x}{\xi}{\alpha}{\beta}\psi_1(x,\xi/h)
$$ 
as long as $|\alpha+\beta|\ge1$. Hence $\{\Psi_1^h\}_{h\in(0,1]}$ is also bounded in $S(1,h^{4/m}dx^2+d\xi^2/\<\xi\>^2)$ and satisfies the support property \eqref{proposition_LP_1_2}. 
\end{proof}

We next recall the square function estimates for the standard Littlewood-Paley projections.

\begin{lemma}[Square function estimates]
\label{lemma_LP_3}
Let $c>1$ and consider a $c$-adic partition of unity: 
$$
\theta_0,\theta\in C_0^\infty(\R^d),\ \supp \theta\subset\{1/c<|x|<c\},\ 0\le \theta_0,\theta\le1,\ \theta_{0}(x)+\sum_{l\ge0} \theta(c^{-l}x)=1.
$$
Then, for any $1<q<\infty$, 
\begin{align*}
\norm{v}_{L^q}
&\approx
\Big|\Big|\Big(|\theta_{0}(D)v|^2+\sum_{l\ge0}|\theta(c^{-l}D)v|^2\Big)^{1/2}\Big|\Big|_{L^q}\\&\approx
\Big|\Big|\Big(|\theta_{0}(x)v|^2+\sum_{l\ge0}|\theta(c^{-l}x)v|^2\Big)^{1/2}\Big|\Big|_{L^q}.
\end{align*}
Moreover, if $2\le q<\infty$ then
\begin{align}
\label{lemma_LP_3_3}
\norm{v}_{L^q}&\lesssim \norm{v}_{L^2}+\Big(\sum_l\norm{\theta(c^{-l}D)v}_{L^q}^2\Big)^{1/2},\\ 
\label{lemma_LP_3_4}
\norm{v}_{L^q}&\lesssim \norm{\theta_{0}(x)v}_{L^q}+\Big(\sum_l\norm{\theta(c^{-l}x)v}_{L^q}^2\Big)^{1/2}.
\end{align}
Here implicit constants depend only on $d$ and $q$. 
\end{lemma}

\begin{proof}
We let $q\in (1,\infty)$ and set $S_1v=(|\theta_{0}(x)v|^2+\sum_{l\ge0}|\theta(c^{-l}x)v|^2)^{1/2}$. Since $\theta^2\le\theta$, 
$$
\Big(\sum_{l\ge0}|\theta(c^{-l}x)|^2\Big)^{1/2}\le
\Big(\sum_{l\ge0}\theta(c^{-l}x)\Big)^{1/2}=
(1-\theta_{0})^{1/2}\le1,
$$
from which we have $\norm{S_1v}_{L^q}\le 2\norm{v}_{L^q}$. Since $\theta(c^{-l}x)\theta(c^{-k}x)=0$ for $|l-k|>2$, we learn by H\"older's inequality that
$$
|\int v_1\overline{v_2}dx|\lesssim \norm{S_1v_1}_{L^q}\norm{S_1v_2}_{L^{q'}},\quad1/q+1/q'=1,
$$
which, together with the upper bound, implies $\norm{v}_{L^q}\lesssim \norm{S_1v}_{L^q}$. Therefore, $\norm{v}_{L^q}\approx \norm{S_1v}_{L^q}$. 

Next we set $S_2v=(|\theta_{0}(D)v|^2+\sum_{l\ge0}|\theta(c^{-l}D)v|^2)^{1/2}$. By the same argument as above, it suffices to check that $\norm{S_2v}_{L^q}\lesssim \norm{v}_{L^q}$. We follow the same argument as that in the proof of \cite[Theorem 0.2.10]{Sogge}. 
For $t\in[0,1]$, define a Fourier multiplier $m_t(D)$ by
$$
m_t(D)=\sum_{l\ge0}r_l(t)\theta(c^{-l}D),
$$ 
where $r_l(t)$ are Rademacher functions. 
Since $|\partial_\xi^\gamma m_t(\xi)|\le C_\gamma c^{-l|\gamma|}\le C_\gamma\<\xi\>^{-|\gamma|}$ with some $C_\gamma>0$ independent of $t$, Mikhlin's multiplier theorem shows that $\norm{m_t(D)v}_{L^q}\le C_q\norm{v}_{L^q}$ for any $q\in(1,\infty)$ with constant independent of $t$. 
On the other hand, Khinchine's inequality implies that
$$
A_q^{-1}\norm{m_t(D)v}_{L^q_t(0,1])}\le\Big(\sum_{l\ge0}|\theta(c^{-l}D)v|^2\Big)^{1/2}\le A_q\norm{m_t(D)v}_{L^q_t(0,1])}.
$$
with some $A_q>0$. 
We therefore conclude that
\begin{align*}
\norm{S_2v}_{L^q(\R^d)}
 &\le C_q \norm{\theta_0(D)v}_{L^q(\R^d)}+C_q\Big|\Big|\Big(\sum_{l\ge0}|\theta(c^{-l}D)v|^2\Big)^{1/2}\Big|\Big|_{L^q(\R^d)}\\
 & \le C_q
 \norm{v}_{L^q(\R^d))}+C_q\Big|\Big|\norm{m_t(D)v}_{L^q_t(0,1])}\Big|\Big|_{L^q(\R^d)}\\
 &\le C_q \norm{v}_{L^q(\R^d)}+C_q\Big|\Big|\norm{m_t(D)v}_{L^q(\R^d)}\Big|\Big|_{L^q_t([0,1])}\\
 &\le C_q \norm{v}_{L^q(\R^d)}. 
\end{align*}

\eqref{lemma_LP_3_3} and \eqref{lemma_LP_3_4} then follows from Minkowski's inequality and Bernstein's inequality since $q\ge2$. 
\end{proof}

We now turn into the proof of Proposition \ref{proposition_LP_1}:

\begin{proof}[Proof of Proposition \ref{proposition_LP_1}]
Set $h=2^{-l}$. We plug 
$
\psi_0(x,D)v
$ 
into \eqref{lemma_LP_3_3} with $c=2$. By virtue of Lemma \ref{lemma_LP_2}, the contribution of the error term $\theta(hD)\wtilde r_N^h(x,D)+h^N r_N^h(x,hD)$ is dominated by $\norm{v}_{L^2}$ provided that $N>d(1/2-1/q)$. We hence have
$$
\norm{\psi_0(x,D)v}_{L^q}\le C_q \norm{v}_{L^2}+\Big(\sum_h\norm{\Psi_0^h(x,hD)v}_{L^q}^2\Big)^{1/2}.
$$
The estimate for $\psi_1(x,D)v$ is verified similarly by using Lemma \ref{lemma_LP_2} and \eqref{lemma_LP_3_4} with $c=2^{2/m}$. 
\end{proof}


\subsection{Local smoothing effects}

We here recall the local smoothing  effects proved by Robbiano-Zuily \cite{RZ}. For $s\in \R$ we set 
$$
e_s(x,\xi):=\left(k_A(x,\xi)+\<x\>^m+L(s)\right)^{s/2}, 
$$
where $k_A(x,\xi)=\frac12g^{jk}(x)(\xi_j-A_j(x))(\xi_k-A_k(x))$ and $L(s)$ is a constant depending on $s$. Then, $e_s \in S(e_s,dx^2/\<x\>^2+d\xi^2/e_1^2)$, that is
\begin{align}
\label{LS_1}
|\dderiv{x}{\xi}{\alpha}{\beta}e_s(x,\xi)|\le C_{\alpha\beta}\ e_{s-|\beta|}(x,\xi)\<x\>^{-|\alpha|}.
\end{align}
Let $E_s=e_s(x,D)$. Then, for any $s\in\R$, there exists $L(s)>0$ such that $E_s$ is a homeomorphism from $\B^{r+s}$ to $\B^r$ for all $r \in \R$, and $E_s^{-1}$ is also a $\Psi$DO with the symbol $\wtilde e_{-s}$ in $S(e_{-s},dx^2/\<x\>^2+d\xi^2/e_1^2)$ (see, \cite[Lemma 4.1]{Doi}). 

We first prepare the following two lemmas:

\begin{lemma}
\label{lemma_LS_1}
For any $s\in\R$, $E_sPE_{s}^{-1}=P+B_s$ with $\norm{B_s-B_s^*}_{L^2\to L^2}\lesssim 1$. 
\end{lemma}

\begin{proof}
We write $P=p_{\text{full}}(x,D)=k_A(x,D)+p_1(x,D)+V(x)$, where 
$$
p_{1}(x,\xi)=-\frac i2\frac{\partial g^{jk}}{\partial x_j}(x)(\xi_k-A_k(x))-\frac i2g^{jk}(x)\frac{\partial A_k}{\partial x_j}(x).
$$ 
A direct computation yields 
$$
\{e_{s},k_A+V\}=-\frac s2e_{s-2}\{\<x\>^m,k_A\}+\frac s2e_{s-2}\{k_A,V\}=e_{s-2}(a_0+a_1),
$$
where $a_0\in S(\<x\>^{m-1}\<\xi\>,g)$, and $a_1(x)$ is independent of $\xi$ and satisfies
$
\partial_x^\alpha a_1=O(\<x\>^{3m/2-1-|\alpha|})
$. 
We similarly obtain
$
\{e_s,p_1\}\in S(e_s\<x\>^{-2},g)
$ 
by \eqref{LS_1}. 
Next, since \eqref{assumption_A_1} and \eqref{LS_1} imply
$$
\partial_\xi^\alpha e_s\in S(e_{s-2},g),\ 
\partial_x^\alpha e_s\in S(e_s,g),\ 
\partial_\xi^\alpha p_\text{full}\in S(1,g),\ 
\partial_x^\alpha p_\text{full}\in S(e_2,g)\ \text{if}\ |\alpha|\ge2,
$$
the expansion formula \eqref{pdo_2} shows that the symbol of $[E_s,P]-\{e_s,p_\text{full}\}(x,D)$ belongs to $S(e_s,g)$. We also learn by \eqref{LS_1} that
\begin{align*}
&\partial_\xi^\alpha(e_{s-2}a_0)\partial_x^\alpha \wtilde e_{-s}=O(e_{-3}\<x\>^{m-1}\<\xi\>+e_{-2}\<x\>^{m-1})=O(\<x\>^{-1}),\\
&\partial_\xi^\alpha(e_{s-2}a_1)\partial_x^\alpha \wtilde e_{-s}=O(e_{-3}\<x\>^{3m/2-1})=O(\<x\>^{-1}),
\end{align*}
for all $|\alpha|\ge1$. Therefore, the symbol of $B_s=[E_s,P]E_{s}^{-1}$ is of the form
$
e_{s-2}\wtilde e_{-s}(a_0+a_1)+a_2$ with some $a_2\in S(1,g)
$. 
By a similar argument, we further have
\begin{align*}
\partial_\xi^\alpha\partial_x^\alpha (e_{s-2}\wtilde e_{-s}a_0)=O(\<x\>^{-1}),\ 
\partial_\xi^\alpha\partial_x^\alpha (e_{s-2}\wtilde e_{-s}a_1)=O(\<x\>^{-1})\quad\text{for}\quad |\alpha|\ge1,
\end{align*}
which, together with the expansion formula \eqref{pdo_3}, imply that the symbol of $B_s-B_s^*$  belongs to $S(1,g)$. The assertion now follows from the Calder\'on-Vaillancourt theorem.
\end{proof}

\begin{lemma}
\label{lemma_LS_2}
For any $s\in\R$ there exists $C_s>0$ such that 
$$
\norm{E_se^{-itP}u_0}_{L^2}\le C_se^{C_s|t|}\norm{E_su_0}_{L^2},\ t\in\R.
$$
\end{lemma}

\begin{proof}
We may assume $t\ge0$ without loss of generality. Set $v(t)=E_se^{-itP}u_0$ and compute
\begin{align*}
\frac{d}{dt}\norm{v(t)}_{L^2}^2
&=\<-i(P+B_s)v(t),v(t)\>+\<v(t),-i(P+B_s)v(t)\>\\
&=-i\<(B_s-B_s^*)v(t),v(t)\>.
\end{align*}
By virtue of the previous lemma, we have 
$
\frac{d}{dt}\norm{v(t)}_{L^2}\lesssim \norm{v(t)}_{L^2}.
$
The assertion then follows from Gronwall's inequality. 
\end{proof}

We now state the local smoothing effects for the propagator $e^{-itP}$. 
\begin{proposition}[The local smoothing effects \cite{RZ}]
\label{proposition_LS_3}
Assume Assumptions \ref{assumption_A} and \ref{assumption_C}. Then, for any $T>0$, $\nu>0$ and $s\in\R$, there exists $C_{T,\nu,s}>0$ such that
\begin{align}
\label{proposition_LS_3_1}
\norm{\<x\>^{-1/2-\nu}E_{s+1/m}e^{-itP}u_0}_{{L^2_TL^2}} \le C_{T,\nu,s} \norm{E_su_0}_{L^2}.
\end{align}
\end{proposition}

\begin{proof}
By time reversal invariance, we may replace the time interval $[-T,T]$ by $[0,T]$ without loss of generality. 
Robbiano-Zuily \cite{RZ} proved the case when $s=0$ only. However, by virtue of Lemmas \ref{lemma_LS_1} and \ref{lemma_LS_2}, general cases can be verified by an essentially same argument. 
Indeed, it was shown in \cite{RZ} that there exists a real valued symbol $\lambda \in S(1,dx^2/\<x\>^{2}+d\xi^2/e_1^2)$ supported in $\{(x,\xi);\<x\>^{m/2}\le C_0 |\xi|\}$, with some $C_0>0$ depending only on $c$ defined in Assumption \ref{assumption_A}, such that, for any $\nu>0$, 
\begin{align}
\label{proof_proposition_LS_3_1}
\{k,\lambda\}(x,\xi)\ge C^{-1}\<x\>^{-1-\nu}e_{1/m}(x,\xi)-C\quad\text{on}\ \R^{2d}
\end{align}
with some $C>0$ depending on $\nu$. Furthermore, 
$
i[P,\Lambda^w]-\{k,\lambda\}^w(x,D)
$ 
is bounded on $L^2$, where $\Lambda^w=\lambda^w(x,D)$ and $a^w(x,D)$ are Weyl quantizations of $\lambda$ and $a$, respectively, see \cite{Martinez} for the definition of the Weyl quantization.  
Setting $v(t)=E_se^{-itP}u_0$ and $N(t)=\<(M+\Lambda^w)v(t),v(t)\>$ with $M=1+\sup|\lambda|$, we have 
\begin{align*}
\frac{d}{dt}N(t)=\<i[P,\Lambda^w]v(t),v(t)\>+\<(B_s-B_s^*)v(t),v(t)\>+\<(\Lambda^wB_s-B_s^*\Lambda^w)v(t),v(t)\>.
\end{align*}
By \eqref{proof_proposition_LS_3_1} and the sharp G\r{a}rding inequality, we obtain
$$
\<i[P,\Lambda^w]v(t),v(t)\>\le -C^{-1}\norm{\<x\>^{-1/2-\nu/2}E_{1/m}v(t)}_{L^2}^2+C\norm{v(t)}_{L^2}^2.
$$
On the other hand, since $\partial_x^\alpha\partial_\xi^\beta \lambda=O(\<x\>^{-|\alpha|}e_{-|\beta|})$, a similar argument as in Lemma \ref{lemma_LS_1} shows that
$
\Lambda^wB_s-B_s^*\Lambda^w=[\Lambda^w,B_s]-(B_s^*-B_s)\Lambda^w
$
is bounded on $L^2$. Therefore, we conclude
\begin{align*}
\frac{d}{dt}N(t)\le -C^{-1}\norm{\<x\>^{-1/2-\nu/2}E_{1/m}v(t)}_{L^2}^2+CN(t)
\end{align*}
since $N(t)\approx \norm{v(t)}_{L^2}^2$. Applying Gronwall's inequality, we obtain the assertion.
\end{proof}


\begin{remark}
Assumption \ref{assumption_C} is only needed for Proposition \ref{proposition_LS_3}.

We also note that if $s=0$ then the estimate \eqref{proposition_LS_3_1} is equivalent to the usual local smoothing effect:
\begin{align}
\label{remark_smoothing_1}
\norm{\<x\>^{-1/2-\nu}\<D\>^{1/m}e^{-itP}u_0}_{L^2_TL^2}\le C_{T,\nu}\norm{u_0}_{L^2},\quad\nu>0. 
\end{align}
Indeed, since 
$$
C_m^{-1}(1+|\xi|^{1/m}+|x|^{1/2})\le (k_A(x,\xi)+\<x\>^m+L(0))^{\frac{1}{2m}}\le C_m(1+|\xi|^{1/m}+|x|^{1/2})
$$
with some $C_m$, we have
\begin{align*}
\<x\>^{-1/2-\nu}(k_A(x,\xi)+\<x\>^m+L(0))^{\frac{1}{2m}}&\le C_m(\<x\>^{-1/2-\nu}\<\xi\>^{1/m}+1),\\
\<x\>^{-1/2-\nu}(k_A(x,\xi)+\<x\>^m+L(0))^{\frac{1}{2m}}&\ge C_m^{-1}\<x\>^{-1/2-\nu}\<\xi\>^{1/m}.
\end{align*}
$\<x\>^{-1/2-\nu}E_{1/m}(\<x\>^{-1/2-\nu}\<D\>^{1/m}+1)^{-1}$ and $\<x\>^{-1/2-\nu}\<D\>^{1/m}(\<x\>^{-1/2-\nu}E_{1/m})^{-1}$ thus are bounded on $L^2$. (Note that, though the above estimates provide the $L^2$-boundedness of the principal term of these operators only, the lower order can be treated similarly.)
In particular, 
\begin{align*}
\norm{\<x\>^{-1/2-\nu}E_{1/m}v}_{L^2}&\le C_m(\norm{\<x\>^{-1/2-\nu}\<D\>^{1/2}v}_{L^2}+\norm{v}_{L^2}),\\
\norm{\<x\>^{-1/2-\nu}E_{1/m}v}_{L^2}&\ge C_m^{-1}\norm{\<x\>^{-1/2-\nu}\<D\>^{1/2}v}_{L^2},\end{align*}
which imply the equivalence between \eqref{proposition_LS_3_1} and \eqref{remark_smoothing_1}. 
\end{remark}


\section{Parametrix construction}
\label{parametrix}
Write 
$
\Gamma^h(L):=\{(x,\xi);\ |\xi|^2+h^2\<x\>^m< L\}
$, 
where $L\ge1$ is a fixed constant so large that $\supp \Psi_k^h\subset\Gamma^h(L)$, $k=0,1$. 
This section is devoted to construct the parametrices of propagators, localized in this energy shell, in terms of the semiclassical Fourier integral operator ($h$-FIO for short). 

\subsection{Classical mechanics}
\label{classical_mechanics}
This subsection discusses the flow generated by $p^h$, that is the solution to the Hamilton system:
\begin{equation}
\left\{
\begin{aligned}
\label{hamilton_system}
\dot{X}_j&=\frac{\partial p^h}{\partial \xi_j}(X,\Xi)= g^{jk}(X)(\Xi_k-hA_k(X)),\\
\dot{\Xi}_j&=-\frac{\partial p^h}{\partial x_j}(X,\Xi)
=-\frac12 \frac{\partial g^{kl}}{\partial x_j}(X)(\Xi_k-hA_k(X))(\Xi_l-hA_l(X))\\
&\quad\quad\quad\quad\quad\quad\quad\ \ +hg^{kl}(X)\frac{\partial A_k}{\partial x_j}(X)(\Xi_l-hA_l(X))
-h^2\frac{\partial V}{\partial x_j}(X),
\end{aligned}
\right.
\end{equation}
with the initial condition
$
(X(0,x,\xi),\Xi(0,x,\xi))=(x,\xi)\in\Gamma^h(L),
$
where $\dot{f}=\partial_t f$. We here suppress the $h$-dependence of the flow for simplicity. 

We first show that the flow is well-defined for $|t|\le \delta h^{-2/m}$ and $(x,\xi)\in\Gamma^h(L)$ with sufficiently small $\delta>0$, which is not obvious since the potential $V$ can blow up in the negative direction like $V(x)\le -C\<x\>^m$. More precisely, we have the following rough a priori bound:

\begin{lemma}
\label{lemma_A_0}
For sufficiently small $\delta_0>0$ there exists $C=C(m,\delta_0)>0$ such that 
$$	
|\Xi(t,x,\xi)|^2+h^2\<X(t,x,\xi)\>^m\le CL,
$$
uniformly in $h\in(0,1]$ and $(t,x,\xi)\in[-\delta_0 h^{-2/m},\delta_0 h^{-2/m}]\times\Gamma^h(L)$. 
\end{lemma}

\begin{proof}
Let $\rho\in C_0^\infty(\R^d)$ be such that $\rho(x)=1$ for $|x|\le 1$ and $\rho(x)=0$ for $|x|\ge2$ and set
\begin{equation}
\begin{aligned}
\label{proof_lemma_A_0_1}
\wtilde V(x)&=\rho\left(\frac{h^{2/m}x}{4L^{1/m}}\right)V(x),\quad 
\wtilde A_{j}(x)=\rho\left(\frac{h^{2/m}x}{4L^{1/m}}\right)A_j(x),\\
\wtilde p^h(x,\xi)&=\frac12g^{jk}(\xi_j-h\wtilde A_j(x))(\xi_k-h\wtilde A_k(x))+h^2\wtilde V(x).
\end{aligned}
\end{equation}
Note that 
\begin{align}
\label{proof_lemma_A_0_2}
\wtilde V=V,\quad\wtilde A=A,\quad\wtilde p^h=p^h\quad\text{if}\quad h^{2/m}|x|< 4^mL
\end{align} 
and that, by Assumption \ref{assumption_A},
$$
h^2|\partial_x^\alpha\wtilde V(x)|+h^2|\partial_x^\alpha\wtilde A_j(x)|^2\le C_\alpha h^{2|\alpha|/m}L^{1-|\alpha|/m},\quad x\in \R^d.
$$
Consider the corresponding Hamilton flow $(\wtilde X(t,x,\xi),\wtilde \Xi(t,x,\xi))$, that is the solution to 
$$
\dot {\wtilde X}=(\partial_\xi \wtilde p^h)(\wtilde X,\wtilde \Xi),\ 
\dot {\wtilde \Xi}=-(\partial_x \wtilde p^h)(\wtilde X,\wtilde \Xi);\quad
(\wtilde X,\wtilde \Xi)|_{t=0}=(x,\xi). 
$$
Since the flow conserves the energy, \emph{i.e.}, $\wtilde p^h(\wtilde X(t),\wtilde \Xi(t))=\wtilde p^h(x,\xi)$, we learn by the ellipticity of $g^{jk}$ that if $(x,\xi)\in\Gamma^h(L)$, then
\begin{align*}
|\wtilde \Xi(t)|^2 
&\le |\wtilde \Xi(t)-h\wtilde A(\wtilde X(t))|^2+|h\wtilde A(\wtilde X(t))|^2\\
&\lesssim \wtilde p^h(x,\xi)+h^2|\wtilde V(\wtilde X(t))|+|h\wtilde A(\wtilde X(t))|^2\\
&\le C_0L, 
\end{align*}
and hence
$
|\dot {\wtilde X}(t)|\lesssim |\wtilde\Xi(t)-h\wtilde A(\wtilde X(t))|\le C_1 L^{1/2}$ with some constants $C_0,C_1>0$ independent of $L$. In particular, 
$$
h^{2/m}|\wtilde X(t)-x|\le C_1L^{1/2}h^{2/m}|t|\le C_1L^{1/2}\delta_0,\ (x,\xi)\in\Gamma^h(L),\ |t|\le\delta_0 h^{-2/m}.
$$
Therefore, for any fixed $m,L$, taking $0<\delta_0<2^{-1/m}C_1^{-1}L^{-1/2+1/m}$ we have
$$
h^{2/m}|\wtilde X(t)|\le h^{2/m}|x|+C_1L^{1/2}\delta_0\le 2L^{1/m}
$$
on $[-\delta_0 h^{-2/m},\delta_0 h^{-2/m}]\times \Gamma^h(L)$. By virtue of this bound and \eqref{proof_lemma_A_0_2} we have that $(\wtilde X(t,x,\xi),\wtilde \Xi(t,x,\xi))$ is a solution to \eqref{hamilton_system} with the initial data $(x,\xi)$. Then the uniqueness theorem of first order ODE shows 
$$
(X,\Xi)\equiv(\wtilde X,\wtilde \Xi)\quad\text{on}\quad[-\delta_0 h^{-2/m},\delta_0 h^{-2/m}]\times\Gamma^h(L).
$$
In particular, $(X,\Xi)$ is well defined on $[-\delta_0 h^{-2/m},\delta_0 h^{-2/m}]\times\Gamma^h(L)$ and the above computations yield the desired bound.  
\end{proof}

We next study more precise behavior of the flow. Set 
$$
\Omega^h(R,L):=\{|x|>R\}\,\cap\,\Gamma^h(L).
$$
Note that $\Omega^h(R,L)\Subset \Omega^h(R',L')$ if $R>R'$ and $L<L'$. 

\begin{lemma}[General case]			
\label{lemma_A_3}
For sufficiently small $0<\delta< \delta_0$, the followings are satisfied:\\
\emph{(1)} For any $h\in(0,1]$, $1\le R\le h^{-2/m}$, $(t,x,\xi)\in[-\delta R,\delta R]\times\Omega^h(R,L)$,
\begin{align}
\label{lemma_A_3_1}
|X(t)-x|+\<x\>|\Xi(t)-\xi|&\le C |t|,\\
\label{lemma_A_3_2}
|\dderiv{x}{\xi}{\alpha}{\beta}(X(t)-x)|+\<x\>|\dderiv{x}{\xi}{\alpha}{\beta}(
\Xi(t)-\xi)|
&\le C_{\alpha\beta}\<x\>^{-|\alpha|}|t|,\quad |\alpha+\beta|\ge1,
\end{align}
where constants $C,C_{\alpha\beta}>0$ may be taken uniformly in $h,R$ and $t$. Moreover, for fixed $h\in(0,1]$, $1\le R\le h^{-2/m}$ and $|t|\le \delta R $, the map 
$
\Lambda(t):(x,\xi) \mapsto (X(t,x,\xi),\xi)
$ 
is diffeomorphic from $\Omega^h(R/2,2L)$ onto its range and satisfies
\begin{align}
\label{lemma_A_3_3}
\Omega^h(R,L)\subset\Lambda(t,\Omega^h(R/2,2L))\subset \Omega^h(R/3,3L),\quad h\in(0,1],\ |t|\le \delta R.
\end{align}
\emph{(2)} If $(Y(t,x,\xi),\xi)$ denotes the inverse map of $\Lambda(t)$, then bounds \eqref{lemma_A_3_1} and \eqref{lemma_A_3_2} still hold with $X(t)$ replaced by $Y(t)$ for  $(t,x,\xi)\in[-\delta R,\delta R]\times\Omega^h(R,L)$. \\
\emph{(3)} The same conclusions also hold with $R=1$ and with $\Omega^h(R,L)$ replaced by $\Gamma^h(L)$, i.e., $X(t)$ and $Y(t)$ satisfy \eqref{lemma_A_3_1} and \eqref{lemma_A_3_2} uniformly in $h\in(0,1]$ and $(t,x,\xi)\in [-\delta,\delta]\times\Gamma^h(L)$.
\end{lemma}

\begin{proof}
We prove the statements (1) and (2) only since the proof of (3) being similar. Suppose that $|t|\le \delta R$ and $(x,\xi)\in \Omega^h(R,L)$. 
By the Hamilton system and Lemma \ref{lemma_A_0},
\begin{align}
\label{proof_lemma_A_3_1}
|\dot{X}(t)|\lesssim 1,\quad
|\dot{\Xi}(t)|\lesssim \<X\>^{-1}+h\<X(t)\>^{m/2-1}+h^2\<X(t)\>^{m-1}\lesssim \<X\>^{-1},
\end{align}
Integrating the first estimate over $-\delta R\le t\le \delta R$, we have
$$
|X(t)-x|\le C|t|\le C\delta R. 
$$
Therefore, taking $\delta>0$ so small that $C\delta<1/2$ we see that
$$
|X(t)|\ge |x|-C\delta R>|x|/2>R/2
$$ 
for $(t,x,\xi)\in (-\delta R,\delta R)\times \Omega^h(R,L)$ which, together with the second estimate of \eqref{proof_lemma_A_3_1}, implies \eqref{lemma_A_1}. 

We next let $|\alpha+\beta|=1$ and differentiate the Hamilton system
$$
\left(\begin{matrix}
\dderiv{x}{\xi}{\alpha}{\beta} \dot X\\
\<x\>\dderiv{x}{\xi}{\alpha}{\beta}  \dot \Xi
\end{matrix}\right)
=
\left(\begin{matrix}
\partial_{x}\partial_\xi p^h & \<x\>^{-1}\partial_{\xi}^2p^h\\
-\<x\>\partial_{x}^2p^h & -\partial_{\xi}\partial_{x}p^h
\end{matrix}\right)
\left(\begin{matrix}
\dderiv{x}{\xi}{\alpha}{\beta} X\\
\<x\>\dderiv{x}{\xi}{\alpha}{\beta} \Xi
\end{matrix}\right)
=
O(\<x\>^{-1})\left(\begin{matrix}
\dderiv{x}{\xi}{\alpha}{\beta} X\\
\<x\>\dderiv{x}{\xi}{\alpha}{\beta} \Xi
\end{matrix}\right)
$$
where we have used the fact that
\begin{align*}
|(\partial_x^2p^h)(X(t),\Xi(t))|
&\lesssim \<x\>^{-2},\quad
|(\partial_\xi^2p^h)(X(t),\Xi(t))|\lesssim1,\\
|(\partial_x\partial_\xi p^h)(X(t),\Xi(t))|&\lesssim \<x\>^{-1}
\end{align*}
on $(-\delta R,\delta R)\times \Omega^h(R,L)$. These bound follow from \eqref{symbols_2} and the bound $X(t)\ge |x|/2\ge \<x\>/4$ on $(-\delta R,\delta R)\times \Omega^h(R,L)$. 
The estimate \eqref{lemma_A_3_2} with $|\alpha+\beta|=1$ then follows from Gronwall's inequality. 

Proofs for higher derivatives follow from an induction on $|\alpha+\beta|$. The inclusion relation \eqref{lemma_A_3_3} and the existence of the inverse of $\Lambda^h(t)$ are verified by a standard argument based on the Hadamard global inverse mapping theorem 

The estimates for $Y(t)$ are verified by differentiating the equality 
$$
x=X(t,Y(t,x,\xi)
$$ 
and  using the estimates for $X(t)$. We refer to \cite[Lemmas A.2, A.4 and A.5]{Mizutani2}  for the details of the proof (see also the proof of the next lemma). 
\end{proof}

For the flat case, we have the following stronger bounds than that in the previous lemma:
\begin{lemma}[Flat case]
\label{lemma_A_1}
Assume that $g^{jk}\equiv\delta_{jk}$ and that either $m\ge4$ or Assumption \ref{assumption_B}. 
Then, for sufficiently small $0<\delta<\delta_0$, the following statements hold:\\
\emph{(1)} For any $(t,x,\xi)\in [-\delta h^{-2/m},\delta h^{-2/m}]\times\Gamma^h(L)$ and $\alpha,\beta\in \Z^{d}_+$, we have
\begin{equation}
\begin{aligned}
\label{lemma_A_1_1}
|\dderiv{x}{\xi}{\alpha}{\beta}(X(t)-x)|&\le C_{\alpha\beta}h^{(2/m)\min(|\alpha,1)}|t|,\\
|\dderiv{x}{\xi}{\alpha}{\beta}(\Xi(t)-\xi)|&\le C_{\alpha\beta}h^{2/m},
\end{aligned}
\end{equation}
where $C_{\alpha\beta}>0$ may be taken uniformly with respect to $x,h$ and $\delta$. \\
\emph{(2)} We denote by $(x,\xi)\mapsto(Y(t,x,\xi),\xi)$ the inverse map of $\Lambda(t)$. Then $Y(t,x,\xi)$ is well-defined for $(-\delta h^{-2/m},\delta h^{-2/m})\times\Gamma^h(L)$ and the bounds \eqref{lemma_A_1_1} still hold with $X(t)$ replaced by $Y(t)$. 
\end{lemma}

In order to prove this lemma we follow the same argument as that in \cite{Yajima2}. We first  introduce the velocity variables
$
v(t,x,\xi)=\Xi(t,x,\xi)-hA(X(t,x,\xi))
$. 
Then $(X(t),v(t))$ solves the following Lagrange equations:
\begin{align}
\label{proof_lemma_hamilton_flow_1_1}
\dot X=v,\ \dot v=hB(X)v-h^2\partial_x V(X);\quad
(X,v)|_{t=0}=(x,\xi-hA(x)),
\end{align}
where $B(x)=(B_{jk}(x))=(\partial_jA_k(x)-\partial_kA_j(x))$. By Lemma \ref{lemma_A_0}, 
\begin{align}
\label{lemma_A_2_0}
|v(t,x,\xi)|\le C_0L
\end{align} on $[-\delta h^{-2/m},\delta h^{-2/m}]\times\Gamma^h(L)$ with some $C_0$ independent of $\delta$ and $h$. The following lemma plays a crucial role.

\begin{lemma}							
\label{lemma_A_2}
Set $I_h=[-\delta h^{-2/m},\delta h^{-2/m}]$. Under the condition in Lemma \ref{lemma_A_1}
\begin{align}
\label{lemma_A_2_1}
|X(t)-x-t\xi-thA(x)|&\le Ch^{2/m}|t|^2,\\
\label{lemma_A_2_2}
|v(t)-\xi-hA(x)|&\le Ch^{2/m}|t|\\
\label{lemma_A_2_3}
\int_{I_h}|h(\partial_x^\alpha B)(X(t))||v(t)|dt&\le C_\alpha h^{2/m},\quad |\alpha|\ge1,
\end{align}
uniformly with respect to $(t,x,\xi)\in I_h\times \Gamma^h(L)$ and  $h\in(0,1]$. 
\end{lemma}

\begin{proof}
Since $h|B(X(t))+h^2|\partial_x V(X(t))|\le Ch^{2/m}$ by Lemma \ref{lemma_A_0}, integrating \eqref{proof_lemma_hamilton_flow_1_1} we have \eqref{lemma_A_2_1} and \eqref{lemma_A_2_2}. 

Let $|\alpha|\ge1$. If $m\ge4$ then, since $\<X(t)\>\le Ch^{-2/m}$ and $m/2-2\ge0$, 
$$
h|(\partial_x^\alpha B)(X(t))|\le h\<X(t)\>^{m/2-2}\le Ch^{4/m}
$$
which implies \eqref{lemma_A_2_3}. Note that if $2\le m<4$ then this argument does not work well since $m/2-2<0$. Indeed, we only have $h|(\partial_x^\alpha B)(X(t))|\le C\<x\>^{-1}$ in general. 

Next we suppose Assumption \ref{assumption_B} and shall prove \eqref{lemma_A_2_3} for $2\le m<4$. We split into two cases. 

Case 1: If $\delta|v(t)|\le h(1+|X(t)|)^{m/2}$ for all $t\in I_h$ then 
$$
h|(\partial_x^\alpha B)(X(t))||v(t)|\le Ch^2\<X(t)\>^{m-2-\mu}\le Ch^{4/m}
$$
which implies \eqref{lemma_A_2_3} since $|I_h|\le 2\delta h^{-2/m}$. 

Case 2: Suppose that $\delta|v(t_0)|> h(1+|X(t_0)|)^{m/2}$ for some $t_0\in I_h$ and set $y=X(t_0)$ and $\eta=v(t_0)=\Xi(t_0)-hA(X(t_0))$. Then, taking $\delta$ (depending on $L$) smaller if necessary we see that
\begin{align}
\label{lemma_A_2_4}
h(1+|X(t)|)^{m/2}<3\delta |\eta|,\quad |\eta|/3\le |v(t)|\le 3|\eta|
\end{align}
for all $t\in I_h$. Indeed, setting 
$
T=\sup\{s\in [0,\delta h^{-2/m}];\ \eqref{lemma_A_2_4}\text{ holds for all }|t|\le s\}, 
$
if we assume $T<\delta h^{-2/m}$ then, by the Lagrange equations, 
\begin{align*}
|\dot v(t)|
&\le Ch^{2/m}|v(t)|+Ch^{2/m}\<X(t)\>^{-1}[h(1+|X(t)|)^{m/2}]^{2-2/m}\\
&\le Ch^{2/m}|v(t)-\eta|+Ch^{2/m}|\eta|+Ch^{2/m}(\delta|\eta|)^{2-2/m}\\
&\le Ch^{2/m}|v(t)-\eta|+Ch^{2/m}|\eta|. 
\end{align*}
for $|t|< T$. Integrating this inequality and using Gronwall's inequality we see that $|v(t)-\eta|\le Ch^{2/m}|\eta||t|\le C\delta|\eta|$ and hence
$$
|\eta|/2\le|v(t)|\le 2|\eta|,\quad|t|<T,
$$
provided that $\delta>0$, which may be taken independently of  $|\eta|$, is small enough. Integrating the Lagrange equations and using the bound $1+|y|\le (h\delta|\eta|)^{-2/m}$, we also have 
$1+|X(t)|\le 1+|y|+2T|\eta|
\le h^{-2/m}(\delta^{2/m}|\eta|^{2/m}+2\delta |\eta|)
$, which, together with the bound \eqref{lemma_A_2_0}, implies
\begin{align*}
h(1+|X(t)|)^{m/2}
&\le \delta|\eta|(1+2\delta^{1-2/m}|\eta|^{1-2/m})^{m/2}\\
&\le \delta|\eta|\left(1+2(\delta C_0L)^{1-2/m}\right)^{m/2}.
\end{align*}
Choosing $\delta$ so small that $1+2(\delta C_0L)^{1-2/m}<2^{2/m}$ we obtain
$$
h(1+|X(t)|)^{m/2}\le 2\delta|\eta|,\quad |t|<T.
$$
By the continuity of the flow with respect to $t$, we see that \eqref{lemma_A_2_3} still holds for $T+\ep$ with some $\ep>0$, which contradicts the definition of $T$. We therefore have $T=\delta h^{-2/m}$. Using \eqref{lemma_A_2_4}, we obtain a lower bound of $\partial_t^2(|X(t)|^2)$: 
\begin{align*}
\partial_t^2(|X(t)|^2)
&=2|v(t)|^2+X(t)\cdot hB(X(t))v(t)-X(t)\cdot h^2\partial_xV(X(t))\\
&\ge |\eta|^2/2-C\delta |\eta|^2\ge |\eta|^2/4,\quad t\in I_h.
\end{align*}
In particular $|X(t)|^2$ is convex, and hence there exists $r\in I_h$ such that
$$
1+|X(t)|^2\ge 1+(t-r)^2|\eta|^2/8+|X(r)|^2\ge 1+(t-r)^2|\eta|^2/8,\ t\in I_h.
$$
Finally, using this bound and the fact that $\<X(t)\>^{m/2-1}\le h^{-1+2/m}$ we conclude
$$
\int_{I_h}|h(\partial_x^\alpha B)(X(t))||v(t)|dt\le C_\alpha h^{2/m}
\int_{I_h}(1+|t-r||\eta|)^{-1-\mu}|\eta|dt\le C_\alpha h^{2/m}.
$$
which completes the proof. 
\end{proof}

\begin{proof}[Proof of Lemma \ref{lemma_A_1}]
We only consider the case $0\le t\le \delta h^{-2/m}$ and the proof of opposite case is analogous. Suppose that $(x,\xi)\in \Gamma^h(L)$. 
\eqref{lemma_A_1_1} with $|\alpha+\beta|=0$ follows from \eqref{lemma_A_2_1} 4 and \eqref{lemma_A_2_2}. Differentiating \eqref{proof_lemma_hamilton_flow_1_1} with $|\alpha+\beta|=1$, we have \begin{align*}
\frac{d}{dt}
\left(\begin{matrix}
\dderiv{x}{\xi}{\alpha}{\beta} X(t)\\
\dderiv{x}{\xi}{\alpha}{\beta} v(t)
\end{matrix}\right)
=
\left(\begin{matrix}
0& 1\\
L^h(t)&hB(X(t))
\end{matrix}\right)
\left(\begin{matrix}
\dderiv{x}{\xi}{\alpha}{\beta} X(t)\\
\dderiv{x}{\xi}{\alpha}{\beta} v(t)
\end{matrix}\right)
\end{align*}
where $L^h(t)$ denotes the matrix with $(j,k)$-component
$$
\sum_{\ell=1}^dh\nabla_{x_k}B_{jl}(X(t))v_l(t)-h^2\partial_{x_j}\partial_{x_k}V(X(t)).
$$
Set $f(t)=\left(\dderiv{x}{\xi}{\alpha}{\beta} X(t),h^{-2/m}\dderiv{x}{\xi}{\alpha}{\beta} v(t)\right)^T$ and 
$$
A(t)=
\left(\begin{matrix}
0& h^{2/m}\\
h^{-2/m}L^h(t)&hB(X(t))
\end{matrix}\right).
$$
\eqref{lemma_A_2_3} then shows $\int_{I_h}|A(t)|dt\le C$ on $\Gamma^h(L)$. On the other hand, $f(t)$ solves 
\begin{align}
\label{proof_Lemma_A_1_1}
\dot f(t)=A(t)(f(t)-f(0))+A(t)f(0)
\end{align}
where $A(t)f(0)=(g_1(t),g_2(t))^T$ is given by
\begin{align*}
g_1(t)&=h^{2/m}\dderiv{x}{\xi}{\alpha}{\beta}(\xi-hA(x)),\\
g_2(t)&=h^{-2/m}L^h(t)\dderiv{x}{\xi}{\alpha}{\beta}x+hB(X(t))\dderiv{x}{\xi}{\alpha}{\beta}(\xi-hA(x),
\end{align*}
which, by virtue of \eqref{lemma_A_2_3}, satisfy
\begin{equation}
\left\{\begin{aligned}
\label{proof_lemma_A_1_1_2}
h^{2/m}|\dderiv{x}{\xi}{\alpha}{\beta}(\xi-hA(x))|&\le C_{\alpha\beta}h^{(2/m)\big(1+\min(|\alpha|,1)\big)},\\
|hB(X(t))\dderiv{x}{\xi}{\alpha}{\beta}(\xi-hA(x))|&\le C_{\alpha\beta}h^{(2/m)\big(1+\min(|\alpha|,1)\big)},\\
\int_{I_h}|h^{-2/m}L^h(t)|dt&\le C_{\alpha\beta}.
\end{aligned}
\right.
\end{equation}
Integrating \eqref{proof_Lemma_A_1_1} and using Gronwall's inequality we have
$$
|f(t)-f(0)|\le \int_{I_h}|A(t)f(0)|dt e^{\int_{I_h}|A(t)|dt}\le C_{\alpha\beta},
$$
which implies 
$$
|\dderiv{x}{\xi}{\alpha}{\beta}(v(t)-\xi-hA(x))|\le C_{\alpha\beta}h^{2/m},\quad 0\le t\le \delta h^{-2/m},
$$ 
and hence 
\begin{align*}
|\partial_t\dderiv{x}{\xi}{\alpha}{\beta}X(t)|
&\le |\dderiv{x}{\xi}{\alpha}{\beta}(v(s)-\xi-hA(x))|+|\dderiv{x}{\xi}{\alpha}{\beta}(\xi-hA(x))|\\
&\le C_{\alpha\beta}h^{(2/m)\min(|\alpha|,1)},\quad 0\le t\le \delta h^{-2/m}.
\end{align*}
Integrating this inequality, we obtain the estimates for $\dderiv{x}{\xi}{\alpha}{\beta}X(t)$:
\begin{align}
\label{proof_lemma_A_1_2}
|\dderiv{x}{\xi}{\alpha}{\beta}(X(t)-x)|\le C_{\alpha\beta}h^{(2/m)\min(|\alpha|,1)}|t|,\quad 0\le t\le \delta h^{-2/m}
\end{align}
Finally, since
\begin{align*}
\partial_x\Xi(t)&=\partial_x(v(t)-hA(x))+h\partial_xA(x)+h(\partial A)(X(t))\partial_xX(t)=O(h^{2/m}),\\
\partial_\xi(\Xi(t)-\xi)&=\partial_x(v(t)-\xi)+h(\partial A)(X(t))\partial_\xi X(t)=O(h^{2/m}),
\end{align*}
we have $|\dderiv{x}{\xi}{\alpha}{\beta}(\Xi(t)-\xi)|\le C_{\alpha\beta}h^{2/m}$ for $0\le t\le \delta h^{-2/m}$. 

The estimates of higher order derivatives are verified by an induction on $|\alpha+\beta|$ and we omit details. 

Next we shall prove (2). Set $F^h:(x,\xi)\mapsto (h^{2/m}x,\xi)$ and 
\begin{align*}
X^h(t,y,\xi):=F\circ X\circ F^{-1}(t,y,\xi)=h^{2/m}X(h^{-2/m}y,\xi).
\end{align*}
By \eqref{lemma_A_1} with $L$ replaced by $2L$, 
$
|\dderiv{y}{\xi}{\alpha}{\beta}(X^h(t)-y)|\le C_{\alpha\beta}h^{2/m}|t| 
$
on $I_h\times F(\Gamma^h(2L))$ for $|\alpha+\beta|\le1$. Hence, if $\delta>0$ is small enough then
\begin{align}
\label{proof_lemma_A_4}
\left|\left|\frac{\partial(X^h(t),\xi)}{\partial(y,\xi)}-\Id_{\R^{2d}}\right|\right|<C\delta<1/2\quad\text{on}\quad I_h\times F(\Gamma^h(2L)).
\end{align}
Then, by the same argument as that in \cite[Lemma A.4]{Mizutani2}, we see that the map $\Psi^h:(y,\xi)\mapsto (X^h(t,y,\xi),\xi)$ is a diffeomorphism from $F(\Gamma^h(2L))$ onto its range for all $t\in I_h$ and that
$$
F(\Gamma^h(L))\subset \Psi^h(F(\Gamma^h(2L))),\quad t\in I_h.
$$
Since $F$ is globally diffeomorphic on $\R^{2d}$, $\Psi:(x,\xi)\mapsto (X(t,x,\xi),\xi)$ is a diffeomorphism from $\Gamma^h(2L)$ onto its range and satisfies
$$
\Gamma^h(L)\subset \Psi(\Gamma^h(2L)),\quad t\in I_h.
$$
Let $(Y(t,x,\xi),\xi):\Gamma^h(L)\to \Gamma^h(2L)$ be the corresponding inverse. Then, by using the inequality
$
x=X(t,Y(t,x,\xi),\xi)
$, we have
$$
|Y(t,x,\xi)-x|=|Y(t,x,\xi)-X(t,Y(t,x,\xi),\xi)|\le\sup_{\Gamma^h(2L)}|x-X(t,x,\xi)|\le C|t|. 
$$
Differentiating this inequality with respect to$\partial_x^\alpha\partial_\xi^\beta$, $|\alpha+\beta|=1$, we also have 
$$
(\partial_x X)(t,Y(t),\xi)\dderiv{x}{\xi}{\alpha}{\beta}(Y(t)-x)=\dderiv{y}{\xi}{\alpha}{\beta}(y-X(t,y,\xi))|_{y=Y(t,x,\xi)}.
$$
SInce $(\partial_x X)(t,Y(t),\xi)$ is invertible and its inverse is uniformly bounded on $I_h\times \Gamma^h(L)$ by \eqref{proof_lemma_A_4}, we see that
$$
|\dderiv{x}{\xi}{\alpha}{\beta}(Y(t)-x)|\le \sup_{\Gamma^h(2L)}|\dderiv{y}{\xi}{\alpha}{\beta}(y-X(t,y,\xi))|\le C_{\alpha\beta}h^{(2/m)|\alpha|}|t|
$$
on $I_h\times \Gamma^h(L)$. The estimates on higher derivatives are verified by an induction on $|\alpha+\beta|$ and we omit details.
\end{proof}

\subsection{Semiclassical paramatrix}
\label{semiclassical_parametrix}
We now turn into the construction of parametrices. We begin with the general case.

\begin{theorem}
\label{theorem_WKB_1}
There exists $\delta>0$ such that, for any $h\in(0,1]$ and $1\le R\le h^{-2/m}$, the following statements are satisfied with constants independent of $h$ and $R$:\\
\emph{(1)} There exists a solution $S^h\in C^\infty((-\delta R,\delta R)\times\R^{2d})$ to the Hamilton-Jacobi equation:
\begin{equation}
\left\{\begin{aligned}
\label{theorem_WKB_1_1}
&\partial_t S^h(t,x,\xi)+p^h(x,\partial_x S^h(t,x,\xi))=0,
\quad (t,x,\xi)\in(-\delta R,\delta R)\times\Omega^h(R/3,3L), \\
&S^h(0,x,\xi)
=x\cdot\xi,\quad
(x,\xi)\in\Omega^h(R/3,3L),
\end{aligned}\right.
\end{equation}
such that 
\begin{align}
\label{theorem_WKB_1_2}
|\dderiv{x}{\xi}{\alpha}{\beta}\left(S^h(t,x,\xi)-x\cdot\xi+t\wtilde p^h(x,\xi)\right)|
\le 
C_{\alpha\beta} \<x\>^{-1-\min(|\alpha|,1)}|t|^2,
\end{align} 
uniformly in $(t,x,\xi)\in(-\delta R,\delta R)\times\R^{2d}$, where $\wtilde p^h$ is given by \eqref{proof_lemma_A_0_1}. \\
\emph{(2)} For any $\chi^h\in S(1,g)$ supported in $\Omega^h(R,L)$ and integer $N\ge0$, there exists a bounded family $\{a^h(t); |t|\le\delta R,\ h\in(0,1]\} \subset S(1,g)$ with $\supp a^h(t)\subset \Omega^h(R/2,2L)$ such that
$$
e^{-itP^h/h}\chi^h(x,hD)=J_{S^h}(a^h)+Q^h(t,N),
$$
where $P^h=h^2P$ and $J_{S^h}(a^h)$ is the $h$-FIO with the phase $S^h$ and the amplitude $a^h$ defined by
$$
J_{S^h}(a^h)f(x)=(2\pi h)^{-d}\int e^{i\left(S^h(t,x,\xi)-y\cdot\xi\right)/h} a^h(t,x,\xi)f(y)dyd\xi,
$$
and the remainder $Q^h(t,N)$ satisfies
\begin{align}
\label{theorem_WKB_1_4}
\sup_{|t|\le \delta R}\norm{Q^h(t,N)}_{L^2\to L^2}\le C_N h^{N-1-2/m}.
\end{align}
Furthermore, if $K^h(t,x,\xi)$ denotes the kernel of $J_{S^h}(a^h)$ then 
\begin{align}
\label{theorem_WKB_1_5}
|K^h(t,x,y)| \lesssim \min\{h^{-d},\ |th|^{-d/2}\},\quad x,\xi \in \R^d,\ h \in (0,1],\ |t| \le \delta R.
\end{align}
\end{theorem}

\begin{proof}
\textbf{Construction of the phase $S^h$}: 
Let $X(t),\Xi(t)$ and $Y(t)$ be as in Lemma \ref{lemma_A_3} with $(R,L)$ replaced by $(R/4,4L)$. Define an action integral $\wtilde S^h$ on $(-\delta R,\delta R)\times\Omega^h(R/4,4L)$ by
$$
\wtilde S^h(t,x,\xi):=x\cdot\xi+\int_0^tL^h(X(s,Y(t,x,\xi),\xi),\Xi(s,Y(t,x,\xi),\xi)ds,
$$
where $L^h=\xi\cdot\partial_\xi p^h-p^h$ is the Lagrangian associated to $p^h$. 
A direct computation yields that $\wtilde S^h$ solves \eqref{theorem_WKB_1_1} and satisfies
$$
(\partial_\xi \wtilde S^h,\partial_x\wtilde  S^h)=(Y(t,x,\xi),\Xi(t,Y(t,x,\xi),\xi)).
$$
Furthermore, the following conservation law holds:
$$
p^h(x,\partial_x \wtilde S^h(t,x,\xi))=p^h(x,\Xi(t,Y(t,x,\xi),\xi))=p^h(Y(t,x,\xi),\xi).
$$
By virtue of Lemma \ref{lemma_A_3} (2), taking $\delta>0$ smaller if necessary we see that 
$$
h^2\<Y(t,x,\xi)\>^m\le 5L,\quad (t,x,\xi)\in(-\delta R,\delta R)\times\Omega^h(R/4,4L)
$$ 
and hence $p^h$ can be replaced by $\wtilde p^h$ since $\wtilde p^h\equiv p^h$ on $\Gamma^h(6L)$. Using Lemma \ref{lemma_A_3} (2) and the fact that 
\begin{align}
\label{proof_theorem_WKB_1_0}
|\partial_x \wtilde p^h(x,\xi)|\lesssim \<x\>^{-1}\ \text{on}\ \Omega^h(R/4,4L),
\end{align} 
we then have, for $|t|\le \delta R$ and $(x,\xi)\in \Omega^h(R/4,4L)$, 
\begin{equation}
\begin{aligned}
\label{proof_theorem_WKB_1_1}
|\wtilde p^h(x,\partial_x \wtilde S^h(t,x,\xi))-\wtilde p^h(x,\xi)|
&\lesssim |Y(t)-x|\int_0^1|(\partial_x\wtilde p^h)(\lambda x +(1-\lambda)Y(t),\xi)|d\lambda\\
&\lesssim \<x\>^{-1}|t|.
\end{aligned}
\end{equation}
It also follows from Lemma \ref{lemma_A_3} (2) that, for $|\alpha+\beta|\ge1$, 
\begin{align}
\label{proof_theorem_WKB_1_2}
|\dderiv{x}{\xi}{\alpha}{\beta}(\wtilde p^h(x,\partial_x \wtilde S^h(t,x,\xi))-\wtilde p^h(x,\xi))|\le C_{\alpha\beta} \<x\>^{-1-\min\{|\alpha|,1\}}|t|.
\end{align}
Integrating with respect to $t$ and using Hamilton-Jacobi equation \eqref{theorem_WKB_1_1}, we see that $\wtilde S^h$ satisfies \eqref{theorem_WKB_1_2} on $\Omega^h(R/4,4L)$. 
Choosing $\psi\in S(1,g)$ so that $\supp \psi \subset\Omega^h(R/4,4L)$ and $\psi\equiv1$ on $\Omega^h(R/3,3L)$, we extend $\wtilde S^h$ to the whole space $\R^{2d}$ as follows:
$$
S^h(t,x,\xi)=x\cdot\xi-t\wtilde p^h(x,\xi)+\psi(x,\xi)\left(\wtilde S^h(t,x,\xi)-x\cdot\xi+t\wtilde p^h(x,\xi)\right).
$$
Then $S^h(t,x,\xi)$ satisfies the assertion. 

\textbf{Construction of the amplitude $a^h$}: 
Let us make the following  ansatz: 
$$
v(t,x)=\frac{1}{(2\pi h)^d}\int e^{i(S^h(t,x,\xi)-y\cdot\xi)/h}a^h(t,x,\xi)f(y)dyd\xi,
$$
where $a^h=\sum_{j=0}^{N-1} h^ja^h_j$. In order to  approximately solve the Schr\"odinger equation
$$
(hD_t+P^h)v(t)=O(h^N);\quad v|_{t=0}=\chi^h(x,hD)u_0,
$$ 
the amplitude should satisfy the following transport equations:
\begin{equation}
\label{transport_WKB}
\left\{
\begin{aligned}
&\partial_t a^h_{0}+\X\cdot \partial_xa_{0}+\Y a^h_{0}=0;\quad a^h_0|_{t=0}=\chi^h,\\
&\partial_t a^h_{j}+\X\cdot \partial_xa_{j}+\Y a^h_{j}+iKa^h_{j-1}=0;\quad a^h_j|_{t=0}=0,\quad 1\le j \le N-1, 
\end{aligned}
\right.
\end{equation}
where $K=-\frac12\partial_jg^{jk}(x)\partial_k$, a vector field $\X$ and a function $\Y$ are defined by
\begin{align*}
\X(t,x,\xi)
&:=(\partial _{\xi} p^h)(x,\partial_x S^h(t,x,\xi)),\\
\Y(t,x,\xi)
&:=[k(x,\partial_x)S^h+p_1^h(x,\partial_xS^h)](t,x,\xi).
\end{align*}
Note that \eqref{theorem_WKB_1_2} and \eqref{symbols_2} imply
\begin{align}
\label{proof_theorem_WKB_1_5}
|\dderiv{x}{\xi}{\alpha}{\beta}\Y(t,x,\xi)|\le C_{\alpha\beta}\<x\>^{-1}
\quad\text{on}\quad
(-\delta R,\delta R)\times \Omega(R/3,3L).
\end{align}
The system \eqref{transport_WKB} can be solved by the standard method of characteristics along the flow generated by $\X(t,x,\xi)$. 
More precisely, let us consider the following ODE
$$
\partial_t z(t,s,x,\xi)=\X(t,z(t,s,x,\xi),\xi);\quad z(s,s)=x.
$$
Then, by virtue of \eqref{proof_theorem_WKB_1_1} and \eqref{proof_theorem_WKB_1_2}, the same argument as that in Subsection \ref{classical_mechanics} yields that there exists $\delta>0$ such that, for any fixed $h\in(0,1]$, $1\le R\le h^{-2/m}$, $z(t,s,x,\xi)$ is well-defined for $t,s\in(-\delta R,\delta R)$ and $(x,\xi)\in \Omega(R/3,3L)$, and satisfies
\begin{align}
\label{proof_theorem_WKB_1_4}
|z(t,s)-x|\le C |t-s|,\quad
|\dderiv{x}{\xi}{\alpha}{\beta}(z(t,s)-x)|\le C_{\alpha\beta}\<x\>^{-1}|t-s|,\quad |\alpha+\beta|\ge1.
\end{align}
For $(t,x,\xi)\in (-\delta R,\delta R)\times \Omega(R/3,3L)$, we then define $a_j$, $j=0,1,...,N-1$,  inductively by
\begin{equation}
\begin{aligned}
\nonumber
a_{0}(t,x,\xi)&=\chi^h(z(0,t,x,\xi),\xi)\exp\left(\int_0^t\Y(s,z(s,t,x,\xi),\xi)ds\right),\\
a_{j}(t,x,\xi)&=-\int_0^t (iKa_{j-1})(s,z(s,t,x,\xi),\xi) \exp\left(\int_u^t\Y(u,z(u,t,x,\xi),\xi)du\right)ds.
\end{aligned}
\end{equation}
It is easy to see from \eqref{proof_theorem_WKB_1_4} and $\supp \chi^h\subset \Omega^h(R,L)$ that $\supp a_j\subset\Omega^h(R/2,2L)$ for all $|t|\le\delta R$. 
Furthermore, taking $\delta>0$ smaller if necessary we see that $a_j$ are smooth on $ \Omega(5R/12,12L/5)$. 
Since $\Omega^h(R/2,2L)\Subset\Omega(5R/12,12L/5)\Subset\Omega(R/3,3L)$, if we extend $a_j$ to the whole space $\R^{2d}$ so that $a_j\equiv0$ outside $\Omega^h(R/2,2L)$, then $a_j$ are still smooth. 
We further learn by \eqref{proof_theorem_WKB_1_4}, \eqref{proof_theorem_WKB_1_5} and the fact $\chi\in S(1,g)$ that $a_j\in S(1,g)$ uniformly with respect to $|t|\le \delta R$ and $h\in(0,1]$. 
Finally, one can check by a direct computation that $a_j$ solve the system \eqref{transport_WKB}. 

\textbf{Justification of the parametrix}: At first, since 
$
|\partial_\xi\otimes\partial_x \wtilde p^h(x,\xi)|\lesssim \<x\>^{-1}$ on $\Gamma^h(3L)
$, 
if $\delta>0$ is small enough then \eqref{theorem_WKB_1_2} implies 
$$
|\partial_\xi\otimes\partial_x S^h(t,x,\xi)-\Id|<1/2\quad\text{for}\ (t,x,\xi)\in(-\delta R,\delta R)\times \Omega^h(R/3,3L).
$$
The standard $h$-FIO theory (cf. \cite{Robert2}) then shows that, for any amplitude $b^h\in S(1,g)$ supported in $\Omega^h(R/2,2L)$ ($\Subset \Omega^h(R/3,3L)$), $h$-FIO $J_{S^h}(b^h)$ is uniformly bounded on $L^2$:
$$
\sup_{|t|\le \delta R}\norm{J_{S^h}(b^h)f}_{L^2}\le C\norm{f}_{L^2},\quad h\in(0,1],\ 1\le R\le h^{-2/m},
$$
where $C>0$ is independent of $h$ and $R$. 

We now prove the remainder estimate \eqref{theorem_WKB_1_4}. We may assume $t\ge0$ without loss of generality since the proof for the opposite case is analogous. By the Duhamel formula, we have
\begin{align*}
e^{-itP^h/h}\chi^h(x,hD)&=J_{S^h}(a^h)+Q^h(t,N),\\
Q^h(t,N)&=-\frac ih\int_0^t e^{-i(t-s)P^h/h}(hD_t+P^h)J_{S^h}(a^h)|_{t=s}ds.
\end{align*}
By \eqref{theorem_WKB_1_1}, \eqref{transport_WKB} and direct computations, we obtain $$(hD_t+P^h)J_{S^h}(a^h)=-ih^NJ_{S^h}(Ka_{N-1}^h).$$ Since $\supp Ka_{N-1}^h\subset\Omega(R/2,2L)$ and $Ka_{N-1}^h\in S(1,g)$, $J_{S^h}(P^ha_{N-1}^h)$ is bounded on $L^2$ uniformly in $h\in(0,1]$, $1\le R\le h^{-2/m}$ and $0\le t\le \delta R$, and \eqref{theorem_WKB_1_4} follows. 

\textbf{Dispersive estimates}: The kernel of $J_{S^h}(a^h)$ is given by
$$
K^h(t,x,y)=(2\pi h)^{-d}\int e^{i\left(S^h(t,x,\xi)-y\cdot\xi\right)/h}a^h(t,x,\xi)d\xi.
$$
If $|t|\le h$, then the assertion is obvious since $a^h$ is compactly supported in $\xi$.  
On the other hand, by virtue of \eqref{theorem_WKB_1_2}, for $(x,\xi)\in\R^{2d}$ we have
$$
\frac{\partial_\xi^2S^h(t,x,\xi)}{t}=-(g^{jk}(x))_{j,k}+O(\delta),\quad h \le |t|\le \delta R,
$$
and $|t^{-1}\dderiv{x}{\xi}{\alpha}{\beta}S^h(t,x,\xi)|\le C_{\alpha\beta}$ if $h \le |t|\le \delta R$ and $|\alpha+\beta|\ge2$. As a consequence, since $g^{jk}(x)$ is uniformly elliptic, the phase function $t^{-1}(S^h(t,x,\xi)-y\cdot\xi)$ has a unique non-degenerate critical point for all $h\le |t|\le \delta R$ and we can apply the stationary phase method to $K^h(t,x,y)$, provided that $\delta>0$ is small enough. Therefore, 
$$
|K^h(t,x,y)|\lesssim h^{-d}|th^{-1}|^{-d/2}\lesssim |th|^{-d/2},\quad h\le|t|\le \delta R,\ x,\xi\in \R^d,\ h\in(0,1],
$$
which completes the proof.
\end{proof}

\begin{remark}
\label{remark_WKB_1}
It can be verified by the same argument and Lemma \ref{lemma_A_3} (3) that
for any symbol $\chi^h\in S(1,g)$ supported in $\Gamma^h(L)$, 
$
e^{-itP^h/h}\chi^h(x,hD)
$ 
can be approximated by a time-dependent $h$-FIO as above if $|t|<\delta$, and in particular obeys the dispersive estimate
$$
\norm{e^{-itP^h/h}\chi^h(x,hD)}_{L^1\to L^\infty}\lesssim \min\{h^{-d},|th|^{-d/2}\},\quad
|t|<\delta,\ h\in(0,1].
$$
\end{remark}

We next state the flat case. 
\begin{theorem}[Flat case]
\label{theorem_WKB_2}
Suppose that $g^{jk}\equiv \delta_{jk}$ and that either $m\ge4$ or Assumption \ref{assumption_B}. Then, there exists $\delta>0$ such that the following statements are satisfied with constants independent of $h\in(0,1]$:\\
\emph{(1)} There exists s solution $S^h\in C^\infty((-\delta h^{-2/m},\delta h^{-2/m})\times\R^{2d})$ to the Hamilton-Jacobi equation:
\begin{equation}
\left\{\begin{aligned}
\nonumber
&\partial_t S^h(t,x,\xi)+p^h(x,\partial_x S^h(t,x,\xi))=0,
\ (t,x,\xi)\in(-\delta h^{-2/m},\delta h^{-2/m})\times\Gamma^h(3L), \\
&S^h(0,x,\xi)
=x\cdot\xi,\ 
(x,\xi)\in\Gamma^h(3L),
\end{aligned}\right.
\end{equation}
such that
$$
|\dderiv{x}{\xi}{\alpha}{\beta}\left(S^h(t,x,\xi)-x\cdot\xi+t\wtilde p^h(x,\xi)\right)|
\le 
C_{\alpha\beta} h^{(2/m)(1+\min\{|\alpha|,1\})}|t|^2
$$
uniformly in $(t,x,\xi)\in(-\delta h^{-2/m},\delta h^{-2/m})\times\R^{2d}$. \\
\emph{(2)} For any $\chi^h\in S(1,g)$ with $\supp\chi^h\subset\Gamma^h(L)$ and integer $N\ge0$, there exists a bounded family $\{a^h(t); t\in (-\delta h^{-2/m},\delta h^{-2/m}),h\in(0,1]\} \subset S(1,g)$ with $\supp a^h(t)\subset \Gamma^h(2L)$ such that
$$
e^{-itH^h/h}\chi^h(x,hD)=J_{S^h}(a^h)+Q^h(t,N),
$$
where the kernel of $J_{S^h}(a^h)$ satisfies dispersive estimates \eqref{theorem_WKB_1_5} for $|t|\le \delta h^{-2/m}$, and the remainder satisfies
$$
\sup_{|t|\le \delta h^{-2/m}}\norm{Q^h(t,N)}_{L^2\to L^2}\le C_N h^{N-1-2/m}.
$$
\end{theorem}

The proof is analogous to the general case and the only difference is to use Lemma \ref{lemma_A_1} and the fact that $\partial_x\wtilde p^h=O(h^{2/m})$ on $\Gamma^h(4L)$ instead of Lemma \ref{lemma_A_3} and \eqref{proof_theorem_WKB_1_0}, respectively. We hence omit the details.


\section{Proof of main theorems}

In this section we prove Theorems \ref{theorem_1} and \ref{theorem_2}. The general strategy is basically same as that of \cite[Section 7]{Mizutani2}. We begin with the following theorem which is a consequence of Theorem \ref{theorem_WKB_1} and Remark \ref{remark_WKB_1}. 

\begin{theorem}
\label{theorem_PT_1}
Fix $L>0$. Then, for sufficiently small $\delta>0$ depending only on $L$, the following statements are satisfied:\\
\emph{(1)} For any $h\in(0,1]$, $1\le R\le h^{-2/m}$ and symbol $\chi^h_R\in S(1,g)$ supported in $\{|x|>R\}\cap\Gamma^h(L)$, 
\begin{align}
\label{theorem_PT_1_1}
\norm{\chi^h_R(x,hD)e^{-itP}\chi^h_R(x,hD)^*}_{L^1 \to L^\infty} \le C_\delta|t|^{-d/2},\quad 0<|t|<\delta hR,
\end{align}
where $C_\delta>0$ may be taken uniformly with respect to $h$ and $R$.\\
\emph{(2)} For $h\in(0,1]$ and any symbol $\chi^h\in S(1,g)$ supported in $\Gamma^h(L)$,  
\begin{align}
\label{theorem_PT_1_2}
\norm{\chi^h(x,hD)e^{-itP}\chi^h(x,hD)^*}_{L^1 \to L^\infty} \le C_\delta|t|^{-d/2},\quad 0<|t|<\delta h. 
\end{align}
\end{theorem}

\begin{proof}
The expansion formula \eqref{pdo_3} and Lemma \ref{lemma_pdo_1}shows that there exists symbols $\chi_{0,R}^h,\chi_1^h\in S(1,g)$ with $\chi_{0,R}^h\subset\supp\chi_R^h$ such that 
$$
\chi_R^h(x,hD)^*=\chi_{0,R}^h(x,hD)\chi^h_1(x,hD)+O_{L^p\to L^q}(h^N)
$$ 
for any $1\le p\le q\le\infty$ and any $N\ge0$. We hence can replace $\chi_R^h(x,hD)^*$ by $\chi_{0,R}^h(x,hD)\chi^h_1(x,hD)$ without loss of generality. Then, the assertion follows from \eqref{theorem_WKB_1_4}, \eqref{theorem_WKB_1_5} and \eqref{pdo_1}. 
\end{proof}

Using Theorem \ref{theorem_PT_1}, the $TT^*$-argument and Duhamel formula, one can obtain following semiclassical Strichartz estimates with inhomogeneous terms. The proof is same as that of \cite[Proposition 7.4]{Mizutani3} (see also \cite[Section 5]{Bouclet_Tzvetkov_1}) and we omit it.

\begin{proposition}
\label{proposition_PT_2}
Let $T>0$ and $(p,q)$ satisfy \eqref{admissible_pair}. Under conditions in Theorem \ref{theorem_PT_1}, there exists $C,C_T>0$ such that
\begin{equation}
\begin{aligned}
\label{proposition_PT_2_1}
\norm{\chi^h_R(x,hD)e^{-itP}u_0}_{L^p_TL^q}
&\le
C_Th\norm{u_0}_{L^2}
+C_T\norm{\chi^h_R(x,hD)u_0}_{L^2}\\
&+C(hR)^{-1/2}\norm{\chi^h_R(x,hD)e^{-itP}u_0}_{L^2_TL^2}\\
&+C(hR)^{1/2}\norm{[P,\chi^h_R(x,hD)]e^{-itP}u_0}_{L^2_TL^2},
\end{aligned}
\end{equation}
\begin{equation}
\begin{aligned}
\label{proposition_PT_2_2}
\!\!\!\!\!\!\!\!\!\!\norm{\chi^h(x,hD)e^{-itP}u_0}_{L^p_TL^q}
&\le 
C_Th\norm{u_0}_{L^2}
+C_T\norm{\chi^h(x,hD)u_0}_{L^2}\\
&+Ch^{-1/2}\norm{\chi^h(x,hD)e^{-itP}u_0}_{L^2_TL^2}\\
&+Ch^{1/2}\norm{[P,\chi^h(x,hD)]e^{-itP}u_0}_{L^2_TL^2},
\end{aligned}
\end{equation}
uniformly with respect to $h\in(0,1]$ and $1\le R\le h^{-2/m}$, where $C>0$ is independent of $T$. 
\end{proposition}

\begin{proof}[Proof of Theorem \ref{theorem_1}]
First of all, Proposition \ref{proposition_LP_1} and Minkowski's inequality show
$$
\norm{e^{-itP}u_0}_{L^p_TL^q}\lesssim \norm{u_0}_{L^2}+\sum_{k=0,1}\Big(\sum_h\norm{\Psi_k^h(x,hD)e^{-itP}u_0}_{L^p_TL^q}^2\Big)^{1/2},\quad p,q\ge2,
$$
with $\Psi_k^h\in S(1,h^{4/m}dx^2+d\xi^2/\<\xi\>^2)$ defined in Proposition \ref{proposition_LP_1} satisfying 
\begin{align*}
\supp \Psi_0^h&\subset\{\<x\>\le C_\ep h^{-2/m},\ C_0^{-1}\le |\xi|\le C_0\},\\
\supp\Psi_1^h&\subset\{C_\ep^{-1}h^{-2/m} \le\<x\>\le C_\ep h^{-2/m},\ |\xi|\le C_0\},
\end{align*}
where $C_\ep,C_0>0$ are independent of $x,\xi$ and $h$. 

We first study $\Psi_1^h(x,hD)e^{-itP}$. The expansion formula \eqref{pdo_2} shows
\begin{align*}
\supp \Sym([P,\Psi_1^h(x,hD)])\subset \supp \Psi_1^h,\quad
\Sym([P,\Psi_1^h(x,hD)])\in S(h^{-1+2/m},g).
\end{align*}
Choose $\wtilde\theta\in C_0^\infty(\R)$ supported away from the origin and $\wtilde\psi_1\in S(1,g)$ supported in $\{(x,\xi);|\xi|^2\le 2\ep\<x\>^m\}$ such that $\wtilde\theta\equiv1$ on $\supp \theta$ and $\wtilde\psi\equiv1$ on $\supp \psi_1$, where $\psi_1$ has been defined in  Subsection \ref{LP}. If we set $\wtilde Psi_1^h(x,\xi)=\wtilde\theta(h^{2/m}x)\wtilde\psi_1(x,\xi/h)$ then $\wtilde\Psi_1^h\in S(1,g)$ with uniform bounds in $h$ and $\wtilde\Psi_1^h\equiv1$ on $\supp \Psi_1^h$. Furthermore, Lemma \ref{lemma_pdo_1} shows
$$
[P,\Psi_1^h(x,hD)]=[P,\Psi_1^h(x,hD)]\wtilde\Psi_1^h(x,hD)+O_{L^2\to L^2}(h).
$$
Therefore, 
\begin{equation}
\begin{aligned}
\label{PT_0}
&\norm{[P,\Psi_1^h(x,hD)]e^{-itP}u_0}_{L^2_TL^2}\\
&\le C
h^{-1/2+1/m}\norm{\wtilde\Psi_1^h(x,hD)e^{-itP}u_0}_{L^2_TL^2}+C_Th\norm{u_0}_{L^2}.
\end{aligned}
\end{equation}
Applying Proposition \ref{proposition_PT_2} to $\Psi_1^h(x,hD)e^{-itP}$ with $R=h^{-2/m}$ and using \eqref{PT_0}, we then obtain
\begin{align*}
&\norm{\Psi_1^h(x,hD)e^{-itP}u_0}_{L^p_TL^q}\\
&\le 
C_Th\norm{u_0}_{L^2}
+C_T\norm{\Psi_1^h(x,hD)u_0}_{L^2}
+Ch^{-1/2+1/m}\norm{\wtilde\Psi_1^h(x,hD)e^{-itP}u_0}_{L^2_TL^2}\\
&\le 
C_Th\norm{u_0}_{L^2}
+C\norm{\theta(h^{2/m}x)\psi_1(x,D)u_0}_{L^2}\\
&\quad+C\norm{\wtilde\theta(h^{2/m}x)\<x\>^{m/4-1/2}\wtilde\psi_1(x,D)e^{-itP}u_0}_{L^2_TL^2},
\end{align*}
where, in the last line, we have used the fact that 
$
h^{-1/2+1/m}\approx \<x\>^{m/4-1/2}$ on $\supp \wtilde \Psi_1^h.
$ 
Combining this estimate with the following the norm equivalence:
$$
\norm{v}_{L^2}^2
\approx\sum_h\norm{\theta(h^{2/m}x)v}_{L^2}^2
\approx\sum_h\norm{\wtilde\theta(h^{2/m}x)v}_{L^2}^2
$$
(which follows from the almost orthogonality of $\theta(h^{2/m}x)$ and $\wtilde \theta(h^{2/m}x)$), we have
\begin{align*}
\sum_h\norm{\Psi_1^h(x,hD)e^{-itP}u_0}_{L^p_TL^q}^2
\le 
C_T\norm{u_0}_{L^2}^2
+C\norm{\<x\>^{m/4-1/2}\wtilde\psi_1(x,D)e^{-itP}u_0}_{L^2_TL^2}^2.
\end{align*}
On the other hand, since 
$
\<x\>^{m/4-1/2}\wtilde\psi_1e_{-1/2+1/m}\in S(1,g)
$, it is easy to see that 
$\<x\>^{m/4-1/2}\wtilde\psi_1(x,D)E_{1/2-1/m}^{-1}$ is bounded on $L^2$. Combining with Lemma \ref{lemma_LS_2} we obtain
\begin{align*}
\norm{\<x\>^{m/4-1/2}\wtilde\psi_1(x,D)e^{-itP}u_0}_{L^2_TL^2}^2
&\le C\norm{E_{1/2-1/m}e^{-itP}u_0}_{L_T^2L^2}^2\\
&\le C\norm{E_{1/2-1/m}u_0}_{L^2}^2\int_0^Tdt\\
&\le C_T\norm{E_{1/2-1/m}u_0}_{L^2}^2, 
\end{align*}
and hence
\begin{align}
\label{PT_1}
\sum_h\norm{\Psi_1^h(x,hD)e^{-itP}u_0}_{L^p_TL^q}^2
\le C_T\norm{E_{1/2-1/m}u_0}_{L^2}^2. 
\end{align}

Next we study $\Psi_0^h(x,hD)e^{-itP}u_0$. Choose a dyadic partition of unity:
$$
\varphi_{-1}(x)+\sum_{0 \le j \le j_h}\varphi(2^{-j}x)=1,\quad x\in \pi_x(\supp \Psi_0^h),
$$
where $j_h\lesssim (2/m)\log(1/h)$ and $\varphi_{-1},\varphi \in C_0^\infty(\R^d)$ with $\supp\varphi_{-1}\subset\{|x|<1\}$ and $\supp \varphi\subset\{1/2<|x|<2\}$. We set $\varphi_j(x)=\varphi(2^{-j}x)$ for $j\ge0$. Since $p,q\ge2$, it follows from Minkowski's inequality that
$$
\norm{\Psi_0^h(x,hD)e^{-itP}u_0}_{L^p_TL^q}^2\\
\le
\sum_{-1\le j\le j_h}\norm{\varphi_j(x)\Psi_0^h(x,hD)e^{-itP}u_0}_{L^p_TL^q}^2.
$$
We here take cut-off functions $\wtilde\varphi_{-1},\wtilde \varphi \in C_0^\infty(\R^d)$ and $\wtilde\Psi_0^h \in S(1,g)$ supported in a small neighborhood of $\supp\varphi_{-1}$, $\supp \varphi$ and $\supp\Psi_0^h$, respectively, so that $\wtilde\varphi_{-1}\equiv1$ on $\supp \varphi_{-1}$, $\wtilde\varphi\equiv1$ on $\supp \varphi$ and $\wtilde\Psi_0^h\equiv1$ on $\supp \Psi_0^h$. 
Set $\wtilde\varphi_j(x)=\wtilde\varphi(2^{-j}x)$ for $j\ge0$. Then, 
$$
\supp\wtilde\varphi_j\wtilde\Psi_0^h\subset\{|x|\approx 2^j,\ |\xi|\approx 1\},\quad
\wtilde\varphi_j\wtilde\Psi_0^h\equiv1\ \text{on}\ \supp \varphi_j\Psi_0^h.
$$ 
Since the symbolic calculus shows
\begin{align*}
&\supp \Sym([P,\varphi_j(x)\Psi_0^h(x,hD)])\subset\supp (\varphi_j\Psi_0^h),\\
&\Sym([P,\varphi_j(x)\Psi_0^h(x,hD)])\in S(2^{-j}h^{-1},g),
\end{align*}
applying Proposition \ref{proposition_PT_2} with $R=2^j$, we learn by a similar argument as above that
\begin{align*}
&\norm{\varphi_j(x)\Psi_0^h(x,hD)e^{-itP}u_0}_{L^p_TL^q}\\
&\le 
C_Th\norm{u_0}_{L^2}
+C_T\norm{\varphi_j(x)\Psi_0^h(x,hD)u_0}_{L^2}\\
&\quad+C(h2^j)^{-1/2}\norm{\wtilde\varphi_j(x)\wtilde\Psi_0^h(x,hD)e^{-itP}u_0}_{L^2_TL^2}\\
&\le 
C_Th\norm{u_0}_{L^2}
+C_T\norm{\varphi_j(x)\Psi_0^h(x,hD)u_0}_{L^2}\\
&\quad +C\norm{\wtilde\varphi_j(x)\<x\>^{-1/2}\<D\>^{1/2}\wtilde\Psi_0^h(x,hD)e^{-itP}u_0}_{L^2_TL^2}.
\end{align*}
The almost orthogonality of $\varphi_j$ and $\wtilde\varphi_j$ then yields
\begin{align*}
&\sum_{-1\le j\le j_h}\norm{\varphi_j(x)\Psi_0^h(x,hD)e^{-itP}u_0}_{L^p_TL^q}^2\\
&\le
C_Th^2\norm{u_0}_{L^2}^2
+C_T\norm{\Psi_0^h(x,hD)u_0}_{L^2}^2\\
&\quad+C\norm{\<x\>^{-1/2}\<D\>^{1/2}\wtilde\Psi_0^h(x,hD)e^{-itP}u_0}_{L^2_TL^2}^2.
\end{align*}
We further obtain by the symbolic calculus that
\begin{equation}
\begin{aligned}
\label{PT_0_0}
&\norm{\<x\>^{-1/2}\<D\>^{1/2}\wtilde\Psi_0^h(x,hD)e^{-itP}u_0}_{L^2_TL^2}\\
&\le C
\norm{\wtilde\Psi_0^h(x,hD)\<x\>^{-1/2}E_{1/2}e^{-itP}u_0}_{L^2_TL^2}+Ch^{1/2}\norm{u_0}_{L^2}.
\end{aligned}
\end{equation}
Indeed, we compute
\begin{equation}
\begin{aligned}
\label{PT_1_0}
&\<x\>^{-1/2}\<D\>^{1/2}\wtilde\Psi_0^h(x,hD)\\
&=\<x\>^{-1/2}\<D\>^{1/2}\cdot\wtilde\Psi_0^h(x,hD)E_{1/2}^{-1}\<x\>^{1/2}\cdot\<x\>^{-1/2}E_{1/2}\\
&=\<x\>^{-1/2}\<D\>^{1/2}E_{1/2}^{-1}\<x\>^{1/2}\cdot\wtilde\Psi_0^h(x,hD)\<x\>^{-1/2}E_{1/2}\\
&+\<x\>^{-1/2}\<D\>^{1/2}[E_{1/2}^{-1}\<x\>^{1/2},\wtilde\Psi_0^h(x,hD)]\<x\>^{-1/2}E_{1/2}.
\end{aligned}
\end{equation}
Since $E_{1/2}^{-1}$ has the symbol $\wtilde e_{-1/2}\in S(e_{-1/2},g)$ and $\<\xi\>^{1/2}
 e_{-1/2}(x,\xi)\lesssim1$, 
 $$
 \<x\>^{-1/2}\<D\>^{1/2}E_{1/2}^{-1}\<x\>^{1/2}
 $$ 
 is bounded on $L^2$. 
On the other hand, by the symbolic calculus, we see that 
\begin{align*}
&\<x\>^{-1/2}\<\xi\>^{1/2}\{e_{-1/2}\<x\>^{1/2},\wtilde\Psi_0^h(x,h\xi)\}\<x\>^{-1/2}e_{1/2}\\
&=O(\<x\>^{-3/2}h\<\xi\>^{1/2}+\<x\>^{-3/2}\<\xi\>^{-1/2})\\
&=O(\<x\>^{-3/2}h^{1/2})
\end{align*}
since $\<\xi\>\approx h^{-1}$ on $\supp \wtilde\Psi_0^h(\cdot,h\cdot)$. Therefore, the last term of \eqref{PT_1_0} is bounded on $L^2$ with the bound $O(h^{1/2})$. We thus obtain \eqref{PT_0_0}. 

Next we choose a smooth cut-off function $\wtilde\theta \in C_0^\infty(\R^d)$ supported away from the origin such that $\wtilde \theta\equiv1$ on $\pi_\xi(\supp \Psi_0^h)$. Lemma \ref{pdo_1} then yields
$$
\norm{\Psi_0^h(x,hD)(1-\wtilde\theta(hD))}_{L^2\to L^q}+\norm{\wtilde\Psi_0^h(x,hD)(1-\wtilde\theta(hD))}_{L^2\to L^q} \le Ch
$$
for $2\le q\le\infty$ and $h\in(0,1]$. We hence may replace $\Psi_0^h(x,hD)$ and $\wtilde\Psi_0^h(x,hD)$ by $\Psi_0^h(x,hD)\wtilde\theta(hD)$ and $\wtilde\Psi_0^h(x,hD)\wtilde\theta(hD)$, respectively. 
Then the $L^2$-boundedness of $\wtilde\Psi_0^h(x,hD)$ and the almost orthogonality of $\wtilde\theta(h\xi)$ imply
\begin{align*}
\sum_h\Big(
h^{2}\norm{u_0}_{L^2}^2+\norm{\Psi_0^h(x,hD)\wtilde\theta(hD)u_0}_{L^2}^2\Big)
&\le C \norm{u_0}_{L^2}^2,\\
\sum_h\norm{\wtilde\Psi_0^h(x,hD)\wtilde\theta(hD)\<x\>^{-1/2}E_{1/2}e^{-itP}u_0}_{L^2_TL^2}^2
&\le C
\norm{\<x\>^{-1/2}E_{1/2}e^{-itP}u_0}_{L^2_TL^2}^2.
\end{align*}
Furthermore, since $\<x\>^{m\nu/2}\le e_{\nu}(x,\xi)$ for any $\nu\ge0$, we have
$$
\norm{\<x\>^{-1/2}E_{1/2}e^{-itP}u_0}_{L^2_TL^2}\le C\norm{\<x\>^{-1/2-m\nu/2}E_{1/2+\nu}e^{-itP}u_0}_{L^2_TL^2}. 
$$
We now take $\nu>0$ arbitrarily and apply Proposition \ref{proposition_LS_3} with $s=1/2-1/m+\nu$ to obtain
\begin{align}
\label{PT_2}
\sum_h\norm{\Psi_0^h(x,hD)e^{-itP}u_0}_{L^p_TL^q}^2\le C_{T,\nu} \norm{E_{1/2-1/m+\nu}u_0}_{L^2}^2.
\end{align}

Summering the estimates \eqref{PT_1} and \eqref{PT_2} we conclude
\begin{align*}
\norm{e^{-itP}u_0}_{L^p_TL^q}
\le C_{T,\nu} \norm{E_{1/2-1/m+\nu}u_0}_{L^2}\le C_{T,\nu}\norm{u_0}_{\B^{\frac12-\frac1m+\nu}}
\end{align*}
for any admissible pair $(p,q)$ with $q<\infty$ and $\nu>0$. 
Finally, if $d\ge3$, then Theorem \ref{theorem_1} can be verified by interpolation the $L^2_TL^{2d/(d-2)}$-estimate with the trivial $L^\infty_TL^2$-estimate. 
For $d=2$, let us fix $\ep>0$ and an admissible pair $(p,q)$ arbitrarily and choose an admissible pair $(p_0,q_0)$ with $2<p_0<p$ and $\nu>0$ so that 
$$
\left(\frac12-\frac1m+\nu\right)\frac{p_0}{p}=\left(\frac12-\frac1m\right)\frac 2p+\ep.
$$ 
Interpolating the $L^{p_0}_TL^{q_0}$-estimate with the $L^\infty_TL^2$-estimate, we have
$$
\norm{e^{-itP}u_0}_{L^p_TL^q}\le C_{T,\nu}\norm{u_0}_{\B^{\frac1p\left(1-\frac1m\right)+\ep}}.
$$
We refer to \emph{e.g.}, \cite{Stein_Weiss} for the interpolation in weighted spaces. 
\end{proof}

\begin{proof}[Proof of Remark \ref{remark_trapping}]
Let $d\ge3$. By the above argument we have
$$
\norm{e^{-itP}u_0}_{L^2_TL^{2d/(d-2)}}
\le C\norm{\<x\>^{-1/2}E_{1/2}e^{-itP}u_0}_{L^2_TL^2}+C\norm{E_{1/2-1/m}u_0}_{L^2}.
$$
Then Lemma \ref{lemma_LS_2} shows that the right hand side is dominated by $C_T\norm{E_{1/2}u_0}_{L^2}$, which proves Remark \ref{remark_trapping} (2) for the endpoint case $(p,q)=(2,2d/(d-2)$. Non-endpoint estimates are verified by the interpolation theorem. 
\end{proof}

Next we prove Theorem \ref{theorem_2}. Hence, in what follows (in this section), we suppose that $H=\frac12(D-A(x))^2+V(x)$ satisfies Assumption \ref{assumption_A}. In this case, we first obtain a slightly long-time dispersive estimate which is better than Theorem \ref{theorem_PT_1} (2).

\begin{theorem}
\label{theorem_PT_3}
Let $I\Subset(0,\infty)$ be an interval and $\delta>0$ small enough. Then, for any $h\in(0,1]$ and symbol $\chi^h\in S(1,g)$ supported in $\Gamma^h(L)$,  
$$
\norm{\chi^h(x,hD)e^{-itH}\chi^h(x,hD)^*}_{L^1 \to L^\infty} \le C_\delta|t|^{-d/2},\quad 0<|t|<\delta h^{1-2/m}. 
$$
\end{theorem}

As in the previous argument, we then have the following:

\begin{proposition}
\label{proposition_PT_4}
Under conditions in Theorem \ref{theorem_PT_3}, we have
\begin{align*}
\norm{\chi^h(x,hD)e^{-itH}u_0}_{L^p_TL^q}
&\lesssim
C_Th\norm{u_0}_{L^2}
+C_T\norm{\chi^h(x,hD)u_0}_{L^2}\\
&+Ch^{-1/2+1/m}\norm{\chi^h(x,hD)e^{-itH}u_0}_{L^2_TL^2}\\
&+Ch^{1/2-1/m}\norm{[H,\chi^h(x,hD)]e^{-itH}u_0}_{L^2_TL^2},
\end{align*}
uniformly in $h\in(0,1]$. 
\end{proposition}

\begin{proof}[Proof of Theorem \ref{theorem_2}]
The proof is analogous to that of Theorem \ref{theorem_1}. The only difference compared to the previous one is the following fact:
\begin{align}
\label{PT2_1}
\Sym([H,\Psi_0^h(x,hD)])=h^{-2}\Sym([H^h,\Psi_0^h(x,hD)])\in S(h^{-1+2/m},g).
\end{align}
(Recall that, in general case, we only have $\Sym([P,\Psi_0^h](x,hD))\in S(h^{-1}\<x\>^{-1},g)$.)
Indeed, since $\Psi_0^h\in S(1,h^{4/m}dx^2+d\xi^2)$ and $\<x\>^{m/2}\lesssim h^{-1}$ on $\supp \Psi_0^h$, we have
\begin{align*}
\{(\xi-hA)^2,\Psi_0^h\}
&=2(\xi-hA)\cdot\partial_x\Psi_0^h-2h^2\partial_xA(\xi-A)\cdot\partial_\xi\Psi_0^h\\
&=O(h^{2/m}+h^{1+2/m}\<x\>^{m/2}+h^2\<x\>^{m-1})
=O(h^{2/m}).
\end{align*}
We similarly obtain $\{p_1^h,\Psi_0^h\}=O(h^{2/m}\<x\>^{-1})$ and $\{h^2V,\Psi_0^h\}=O(h^{2/m})$. Therefore
$$
\Sym([H^h,\Psi_0^h(x,hD)])=\frac hi\{(\xi-hA)^2/2+hp_1^h+h^2V,\Psi_0^h\}+O(h^2)=O(h^{1+2/m})
$$
in $S(1,g)$. Applying Proposition \ref{proposition_PT_4} to $\Psi_0^h(x,hD)e^{-itH}u_0$ and using the same argument as above and \eqref{PT2_1}, we have
\begin{align*}
&\norm{\Psi_0^h(x,hD)e^{-itH}u_0}_{L^2_TL^{\frac{2d}{d-2}}}\\
&\le 
C_T\norm{\Psi_0^h(x,hD)u_0}_{L^2}+C_Th\norm{u_0}_{L^2}
+C\norm{\wtilde\Psi_0^h(x,hD)E_{1/2-1/m}e^{-itH}u_0}_{L^2_TL^2}.
\end{align*}
By the almost orthogonality of the $\xi$-support of $\Psi_0^h(x,h\xi)$ and Lemma \ref{lemma_LS_2}, we then conclude
\begin{align*}
\sum_h\norm{\Psi_0^h(x,hD)e^{-itH}u_0}_{L^2_TL^{\frac{2d}{d-2}}}^2
&\le
C_T\norm{u_0}_{L^2}^2
+C\norm{E_{1/2-1/m}e^{-itH}u_0}_{L^2_TL^2}^2\\
&\le 
C_T\norm{E_{1/2-1/m}u_0}_{L^2}^2.
\end{align*}
which, together with the estimates \eqref{PT_1} and Lemma \ref{proposition_LP_1}, implies
$$
\norm{e^{-itH}u_0}_{L^2_TL^{\frac{2d}{d-2}}}\le C_T\norm{E_{1/2-1/m}u_0}_{L^2}.
$$
Finally, the assertion follows from an interpolation with the $L^\infty_TL^2$-estimate.
\end{proof}

\appendix

\section{The case with growing potentials}
Throughout this appendix we assume that $H=\frac12(D-A(x))^2+V(x)$ satisfies Assumption \ref{assumption_A} and \eqref{assumption_V}. We here prove Corollary \ref{remark_1} with a simpler proof than that of Theorem \ref{theorem_2}. The main point is the following square function estimates:

\begin{proposition}
\label{proposition_Appendix_1}
Consider a $4$-adic partition of unity on $[0,\infty)$:
\begin{align*}
f_0,f\in C_0^\infty(\R),\ \supp f\subset [1/4,4],\ 0\le f_0,f\le1,\ f_0(\lambda)+\sum_h f(h^2\lambda)=1,\ \lambda\ge0.
\end{align*}
Then, for any $1<q<\infty$, we have the square function estimates:
\begin{align}
\label{proposition_Appendix_1_1}
\norm{v}_{L^q}
\approx
\Big|\Big|\Big(|f_0(H)v|^2+\sum_h|f(h^2H)v|^2\Big)^{1/2}\Big|\Big|_{L^q}.
\end{align}
In particular, if $2\le q<\infty$ then 
\begin{align}
\label{proposition_Appendix_1_2}
\norm{v}_{L^q}
\lesssim
\Big(\norm{f_0(H)v}_{L^q}^2+\sum_{h}\norm{f(h^2H)v}_{L^q}^2\Big)^{1/2}.
\end{align}
Here, implicit constants are independent of $v$. 
\end{proposition}

\begin{proof}
Since $V\ge0$ and $V\in L^1_{loc}$, the heat kernel of the Schr\"odinger semigroup $e^{-tH}$ satisfies the upper Gaussian bound (cf. Simon \cite{Simon}):
$$
|\partial_x^j e^{-tH}(x,y)|\le C_d t^{-(d+j)/2}e^{-c|x-y|^2/t},\quad t>0,\ j=0,1.
$$
Then, a general theorem by Zheng \cite{Zheng} implies the square function estimates \eqref{proposition_Appendix_1_1}. \eqref{proposition_Appendix_1_2} is an immediate consequence of \eqref{proposition_Appendix_1_1} and Minkowski's inequality since $q\ge2$. 
\end{proof}

Thanks to the positivity of $V$, one can also prove an approximation theorem of the spectral multiplier $f(h^2H)$ in terms of $h$-$\Psi$DO:

\begin{proposition}			
\label{proposition_FC_1}
\emph{(1)} Let $f\in C_0^\infty(\R)$ with $\supp f\Subset (0,\infty)$. Then, for any $N\ge0$, there exists a bounded family $
\{\chi^h\}_{h\in(0,1]}\subset S(1,g)$ with $\supp \chi^h\subset\supp (f\circ p^h)$ such that
$$
\norm{f(h^2H)-\chi^h(x,hD)}_{L^2\to L^q} \le C_{qN}h^{N-d(1/2-1/q)},\quad 2\le q\le\infty,
$$
uniformly in $h\in(0,1]$. In particular, $f(h^2H)$ is bounded from $L^2$ to $L^q$ with the bounds:
$$
\norm{f(h^2H)}_{L^2\to L^q} \le C_{q}h^{-d(1/2-1/q)},\quad 2\le q\le\infty,\ h\in(0,1]. 
$$
\emph{(2)} Let $f_0\in C_0^\infty(\R)$. Then,  
$
\norm{f_0(H)}_{L^2\to W^{s,q}} \le C_{qs}
$ 
for any $2\le q\le\infty$ and $s\ge0$.
\end{proposition}

\begin{proof}
The proof is essentially same as in the case when $A\equiv V\equiv0$ (see, \emph{e.g.}, \cite{Bouclet_Tzvetkov_1}). 
Thus we only outline the proof. 
By \eqref{assumption_V}, we have the following ellipticity:
\begin{align}
\label{proof_proposition_FC_1_1}
p^h(x,\xi)\approx |\xi|^2+h^2\<x\>^m,
\end{align}
where the implicit constants are independent of $h\in (0,1]$, which can be verified as follows: 
if $|\xi|^2\ge Ch^2\<x\>^m$ for sufficiently large $C>0$ then $p^h\ge |\xi|^2-h^2|A|^2+h^2V\gtrsim |\xi|^2+h^2\<x\>^m$; otherwise, by Assumption \ref{assumption_A} (1), $p^h\ge h^2V\gtrsim |\xi|^2+h^2\<x\>^m$ since $|\xi|^2\lesssim h^2\<x\>^m$. 
The upper bound is obvious. 
Using this bound and \eqref{symbols_2}, we see that
$$
\left|\dderiv{x}{\xi}{\alpha}{\beta}\left(\frac{1}{p^h(x,\xi)-z}\right)\right|\le C_{\alpha\beta}\<x\>^{-|\alpha|}\<\xi\>^{-|\beta|}|\Im z|^{-1-|\alpha+\beta|}, 
$$
uniformly in $x,\xi\in\R^d$ and $h\in(0,1]$, and locally uniformly in $z\in \C\setminus \R$. 
Then we can follow the standard argument (see, \emph{e.g.}, \cite{Robert2, BGT,Bouclet_Tzvetkov_1}) to construct the semiclassical approximation of the resolvent $(h^2H-z)^{-1}$ which has the following form:
\begin{align}
\label{proof_proposition_FC_1_2}
(h^2H-z)^{-1}=\sum_{0\le j \le N-1} h^jq_j^h(z,x,hD)+h^N r^h_N(z,x,hD)(h^2H-z)^{-1},
\end{align}
where $q_j^h\in S(\<x\>^{-j}\<\xi\>^{-j},g)$ are of the forms
$$
q_0^h(z,x,\xi)=\frac{1}{p^h(x,\xi)-z},\ q_j^h(z,x,\xi)=\sum_{0 \le k \le j} \frac{q_{jk}^h(x,\xi)}{(p^h(x,\xi)-z)^{1+j+k}},\ j\ge1,
$$ 
with $q_{jk}^h\in S(\<x\>^{-j}\<p^h(x,\xi)\>^{N_j(k)},g)$ with some integer $N_j(k)$. Moreover, the remainder $r_N^h$ belongs to $S(\<x\>^{-N}\<\xi\>^{-N},g)$ with the bounds 
$$
|\dderiv{x}{\xi}{\alpha}{\beta}r_{N}^h(z,x,\xi)|
\le C_{N\alpha\beta}
\<x\>^{-N-|\alpha|}\<\xi\>^{-N-|\beta|}|\Im z|^{-2N-1-|\alpha+\beta|},
$$
uniformly in $x,\xi\in\R^d$ and $h\in(0,1]$, and locally uniformly in $z\in \C\setminus \R$. In particular, if $N>d$ then $r_N^h(z,x,hD)$ is bounded from $L^2$ to $L^q$ with the bounds
\begin{align}
\label{proof_proposition_FC_1_3}
\norm{r_N^h(z,x,hD)}_{L^2 \to L^q}\lesssim h^{-d(1/2-1/q)}|\Im z|^{-n(N,q)}
\end{align}
for $q\in[2,\infty]$, $h\in(0,1]$ and $z\in \C\setminus\R$, where $n(N,q)$ is a positive number depending on $N$ and $q$. 

We now plug the approximation \eqref{proof_proposition_FC_1_2} into the well-known Helffer-Sj\"ostrand formula \cite{Helffer_Sjostrand}:
$$
f(h^2H)=-\frac{1}{2\pi i}\int_\C \frac{\partial\wtilde{f}}{\partial \overline{z}}(z)(h^2H-z)^{-1}dz\wedge d\overline{z},
$$ 
where $dz\wedge d\bar{z}=-2idudv$ with $z=u+iv$ and $\wtilde{f}$ is an almost analytic extension of $f$. Note that, since $f\in C_0^\infty(\R)$,  $\wtilde f$ is also compactly supported and, for any $M\ge0$, 
\begin{align}
\label{proof_proposition_FC_1_4}
|\partial_{\bar{z}}\wtilde{f}(z)|\le C_M|\Im z|^M
\end{align}
with some $C_M>0$. Let us set the symbol $\chi^h$ defined by 
$$
\chi^h=\sum_{j=0}^{N-1}h^j\chi_j^h\ \text{with}\ \chi_0^h=f\circ p^h,\ \chi_j^h=\sum_{k=0}^j\frac{(-1)^{j+k}}{(j+k)!}q_{jk}^h\cdot f^{(j+k)}\circ p^h,\ 1\le j\le N-1.
$$
Then, $\{\chi^h\}_{h\in(0,1]}$ is bounded in $S(1,g)$ and $\supp \chi^h\subset \supp f\circ p^h$. Moreover, taking $N>d$ we learn by \eqref{proof_proposition_FC_1_3} and \eqref{proof_proposition_FC_1_4} that the remainder
\begin{align*}
R^h_N&:=f(h^2H)-\chi^h(x,hD)\\
&=-\frac{h^N}{2\pi i}\int_{\supp \wtilde f}\frac{\partial \wtilde{f}}{\partial \bar{z}}(z)r_{N}^h(z,x,D)(h^2H-z)^{-1}dz\wedge d\bar{z}
\end{align*}
satisfies
\begin{align*}
\norm{R^h_N}_{L^2\to L^q}
&\le C_{Nq}h^{N-d(1/2-1/q)}\int_{\supp \wtilde f}|\Im z|^{M-n(N,q)-1}dz\wedge d\bar{z}\\
&\le C_{Nq}'h^{N-d(1/2-1/q)},
\end{align*}
provided that $M\ge n(N,q)+1$. We complete the proof for the high energy part. 

The assertion for the low energy part is also verified by the same argument as above with $h=1$. 
\end{proof}

\begin{remark}
Assume (for simplicity) that $A\equiv0$. It is easy to see that $[x,H]$ is bounded in $x$. Since $H$ is elliptic, we then learn by a standard commutator argument (see, \emph{e.g.}, \cite[Section 2]{Bouclet_Tzvetkov_1}) that there exists $N>0$ such that $(H-z)^{-N}$ is bounded on $L^q$ for any $q\in[1,\infty]$ with norm dominated by a power of $\<z\>|\Im z|^{-1}$. Therefore, using the same argument as that in \cite[Section 2]{Bouclet_Tzvetkov_1} one can prove in this case that $f(h^2H)$ is bounded from $L^{q'}$ to $L^q$ for any $1\le q'\le q\le\infty$ with the norm of order $h^{-d(1/{q'}-1/q)}$. However, since the $L^2 \to L^q$ boundedness is sufficient to study \emph{local-in-time} Strichartz estimates, we do not write here the precise statement.
\end{remark}

\begin{proof}[Proof of Corollary \ref{remark_1}]
Let $f$ be as in Lemma \ref{proposition_Appendix_1} and choose $F\in C_0^\infty(\R)$ so that $F\equiv1$ on $\supp f$ and $\supp F\Subset(0,\infty)$. We learn by Proposition \ref{proposition_FC_1} (1), Theorem \ref{theorem_PT_3} and the $TT^*$-argument that
$$
\norm{F(h^2H)e^{-itH}u_0}_{L^p_TL^q}\le C_T h^{	(1-2/m)/p}\norm{u_0}_{L^2}
$$
for any admissible pair $(p,q)$. Since $f(h^2H)=f(h^2H)F(h^2H)$ by the spectral decomposition theorem, the estimate \eqref{proposition_Appendix_1_2}, Proposition \ref{proposition_FC_1} (2) and the above estimates imply that
\begin{align*}
\norm{e^{-itH}u_0}_{L^p_TL^q}
&\le C_T \norm{u_0}_{L^2}
+C\Big(\sum_h\norm{e^{-itH}F(h^2H)f(h^2H)u_0}_{L^p_TL^q}^2\Big)^{1/2}\\
&\le C_T\norm{u_0}_{L^2}+C_T\Big(\sum_hh^{2(1-2/m)/p}\norm{f(h^2H)u_0}_{L^2}^2\Big)^{1/2}\\
&\le C_T\norm{\<H\>^{(1/2-1/m)/p}u_0}_{L^2},
\end{align*}
provided that $(p,q)$ is admissible and $q<\infty$. 
\end{proof}


\end{document}